\documentclass[11pt,reqno,a4paper]{amsart}

\usepackage{fancyhdr}
\usepackage{mathrsfs}
\usepackage{dsfont,pifont} 
\usepackage{amssymb,amsmath}
\usepackage{graphicx}
\usepackage{cancel}
\usepackage{nomencl}
\usepackage{xfrac} 

\makenomenclature 

\usepackage{tikz,tikz-cd}
\usepackage{relsize}
\usetikzlibrary{calc}

\usepackage{hyperref}\hypersetup{colorlinks=true, linkcolor=purple, citecolor=blue}

\numberwithin{equation}{section}

\newtheoremstyle{mytheorem}
{10pt}
{10pt}
{\it}
{\parindent}
{\bf}
{.}
{ }
{\thmnumber{#2.~}\thmname{#1}\thmnote{~\rm#3}}

\newtheoremstyle{mytheorem*}
{10pt}
{10pt}
{\it}
{\parindent}
{\bf}
{.}
{ }
{\thmname{#1}\thmnote{~\rm#3}}

\newtheoremstyle{myremark}
{10pt}
{10pt}
{\rm}
{\parindent}
{\bf}
{.}
{ }
{\thmnumber{#2.~}\thmname{#1}\thmnote{~\rm#3}}

\newtheoremstyle{myparagraph}
{10pt}
{10pt}
{\rm}
{\parindent}
{\bf}
{.}
{ }
{\thmnumber{#2.~}\thmname{#1}\thmnote{#3}}

\theoremstyle{mytheorem}
\newtheorem{theorem}[subsection]{Theorem}
\newtheorem{lemma}[subsection]{Lemma}
\newtheorem{corollary}[subsection]{Corollary}

\theoremstyle{mytheorem*}
\newtheorem{theorem*}[]{Theorem}

\theoremstyle{myremark}
\newtheorem{remark}[subsection]{Remark}
\newtheorem{definition}[subsection]{Definition}
\newtheorem{example}[subsection]{Example}

\theoremstyle{myparagraph}

\newtheorem*{parag*}{}

\makeatletter
\def\@secnumfont{\sc}
\def\section{\@startsection%
{section}
{1}
\z@{1.5\linespacing\@plus .2\linespacing}
  {.7\linespacing}
  {\normalfont\sc\centering}}
\makeatother

\makeatletter
\renewenvironment{proof}[1][\proofname]{\par 
  \pushQED{\qed}%
  \normalfont \topsep10\p@\@plus6\p@\relax 
  \trivlist 
  \item[\hskip\labelsep 
    \bfseries 
    #1\@addpunct{.}]\ignorespaces 
}{%
  \popQED\endtrivlist\@endpefalse 
} 
\providecommand{\proofname}{Proof}
\makeatother

\newcommand{\R}{\mathbb{R}}
\newcommand{\Haus}{\mathscr{H}}

\newcommand{\bd}{\partial}

\newcommand{\uno}{\mathds{{1}}}
\newcommand{\N}{\mathbb{N}}

\newcommand{\Z}{\mathbb{Z}}
\newcommand{\D}{\mathcal{D}}
\newcommand{\Gfr}{\mathfrak G}
\newcommand{\gfr}{{{\mathfrak g}}}
\newcommand{\gx}{{\mathpzc{g}_{L}}}
\newcommand{\Ll}{\mathcal L}
\newcommand{\ur}{\mathrm u}
\newcommand{\sr}{\mathrm s}
\newcommand{\qr}{\mathrm q}
\newcommand{\prom}{\mathrm p}
\newcommand{\rv}{\mathrm v}
\newcommand{\rr}{\mathrm r}
\newcommand{\reu}{} 
\newcommand{\red}{} 

\DeclareMathAlphabet{\mathpzc}{OT1}{pzc}{m}{it}

\DeclareMathOperator{\acc}{Acc}
\DeclareMathOperator{\setttgh}{SettTanh}
\DeclareMathOperator{\infl}{Inf{}l}
\DeclareMathOperator{\ri}{r.i.}
\DeclareMathOperator{\clos}{clos}
\DeclareMathOperator{\BV}{BV}
\DeclareMathOperator{\pt}{\bd_{\mathit t}}
\DeclareMathOperator{\px}{\bd_{\mathit x}}
\DeclareMathOperator{\ddt}{\frac{d}{d\mathit t}}
\DeclareMathOperator{\dds}{\frac{d}{d\mathit s}}
\newcounter{stepnb}
\newcounter{substepnb}
\newcommand{\firststep}{\setcounter{stepnb}{0}}
\newcommand{\firstsubstep}{\setcounter{substepnb}{0}}
\newcommand{\step}[1]{\vspace{.3\baselineskip}\emph{\addtocounter{stepnb}{1}\noindent\arabic{stepnb}: #1.}  }
\newcommand{\substep}[1]{\vskip.3\baselineskip\emph{\addtocounter{substepnb}{1} \arabic{stepnb}.\arabic{substepnb}: #1.}  }

\begin{document}


\thispagestyle{empty}
~\vskip -1.1 cm

	%
	%
%

	%
	%
{\centering\Large\bf
Eulerian, Lagrangian and broad continuous solutions 
\\
to a balance law with non convex flux II
\par
}

\vspace{.7 cm}

	%
	%
{\centering\sc 
Giovanni Alberti, Stefano Bianchini, Laura Caravenna
\par
}

\vspace{.6 cm}

	%
	%

\vspace{.8 cm}

{\rightskip 1 cm
\leftskip 1 cm
\parindent 0 pt
\footnotesize

	%
	%
{\sc Abstract.}
We consider a \emph{continuous} solution $u$ of the balance law 
\[
\pt u + \px (f(u)) = \gfr
\] 
in one space dimension, where the flux function $f$ is of class $C^2$ and the source term $\gfr$ is bounded.
This equation admits an Eulerian intepretation (namely the distributional one) 
and a Lagrangian intepretation (which can be further specified).
Since $u$ is only continuous, these interpretations do not necessessarily agree; 
moreover each interpretation naturally entails a different equivalence class for the source term $\gfr$. 
In this paper we complete the comparison between these notions of solutions 
started in the companion paper \cite{file1ABC}, and analize in detail the relations between 
the corresponding notions of source term.
\par
\medskip\noindent
{\sc Keywords:} Balance laws, Eulerian formulation, Lagrangian formulation.
\par
\medskip\noindent
{\sc MSC (2020):} 
35L60, 37C10, 58C20, 76N10
\par
}

\tableofcontents

\section{Introduction and statements of the main results}
\label{s1}
In this paper we consider the balance law 
\begin{equation}
\label{EE}
\pt u + \px (f(u)) = g
\, , 
\end{equation}
where the flux function $f:\R\to\R$ is of class $C^2$, 
the source term $g:\R^2\to\R$ is bounded and Borel regular, 
and the solution $u:\R^2\to\R$ is continuous.

\smallskip

If $u$ is of class $C^1$ and~\eqref{EE} holds
in a pointwise sense for every $t,x$ and if $i_{\gamma}$ is an integral curve 
of class $C^1$ of the vectorfield $(1,f'(u))$, the chain rule provides%
\begin{equation}
\label{LE}
(u(i_{\gamma}))' = g(i_{\gamma})
\, ,
\end{equation}
that is, $u(i_{\gamma})$ is a primitive of $g(i_{\gamma})$ whenever $u$ is a classical solution of ~\eqref{EE} and $i_{\gamma}$ a characteristic curve.
\\
Conversely, if $u$ is a function of class $C^1$ and~\eqref{LE}
holds for every characteristic curve $i_{\gamma}$, then $u$ is 
a classical solution of~\eqref{EE}.

\smallskip
We are interested now in the case when $u$ is \emph{continuous}.
What are the possible notions of solutions of the balance law~\eqref{EE}?
\\
On the one hand, $u$ could be an \emph{Eulerian solution}, when~\eqref{EE} is satisfied in a distributional sense: for all Lipschitz continuous functions $\varphi$, compactly supported, then
\[
-\iint
\left(u \pt \varphi +f(u) \px \varphi \right)=\iint \varphi g\,.
\]
\\
On the other hand, $u$ could be a \emph{Lagrangian solution}, when~\eqref{LE} holds along a monotone selection of characteristic curves that provides a change of variables.
This notion of solution could be further strengthened requiring that~\eqref{LE} holds for every characteristic curve $i_{\gamma}$, in which case we call $u$ \emph{Broad solution}.

\smallskip

Depending on what is the meaning of~\eqref{EE}, the source term $g$ changes its meaning. For Eulerian solutions, it is relevant which \emph{distribution $g$ itself identifies}: it will then be irrelevant to change values of $g$ on subsets which are $\Ll^{2}$-negligible in the plane. For Lagrangian solutions, the relevant objects are the \emph{distributions that $g$ identifies} when restricted \emph{along the characteristic curves that are selected} for the change of variables.
For Broad solutions, what matters are \emph{all the distributions that $g$ identifies when restricted along all characteristic curves}.
\\
From now on we denote the source term $g$ with the more peculiar symbol $\gfr$ in order to stress that we consider a pointwise defined function:
the restriction of $\gfr$ on characteristic curves properly defines a distribution, since $\gfr$ is Borel regular. 

\smallskip

In the companion paper \cite{file1ABC} we established the equivalence among different notions of solutions to~\eqref{EE} for general smooth fluxes. We eventually proved that continuous solutions are Kruzkov iso-entropy solutions, which yields uniqueness for the Cauchy problem.
The source term was an active player: the equivalences, or the reductions, were not stated for a given Borel regular source term $\gfr$, they were rather established framing the source term as the proper meaningful object identified by the corresponding notion of solution.
\\
For example, under the sharp assumption that the set of inflection points of the flux $f$ is negligible, we proved that Eulerian solutions to the equation $\pt u + \px (f(u)) = g$ are Broad solutions to the infinitely dimensional system of ODEs $(u(i_{\gamma}))' = \widehat\gfr(i_{\gamma})$, indexed by all characteristics curves $i_{\gamma}$, see \cite[Theorem 37]{file1ABC}.
What was left to a separate analysis is the sharpness of the assumption on inflection points, and the correspondence of the Eulerian source term $g$ and the Broad source term $\widehat\gfr$.

\smallskip

There are striking features that we discover establishing relations among source terms in different formulations.
In particular, we emphasize:
\begin{itemize}
\item[\S~\ref{Ss:Notlipanlongchar}]
Eulerian solutions might fail to be Broad solutions if inflection points are not negligible:  we exhibit a continuous Eulerian solution which is not Lipschitz continuous on some characteristic curves, although it must be Lipschitz continuous on a suitable selection of them, being a Lagrangian solution by \cite[Corollary 46]{file1ABC}.
\emph{Our counterexample has a source term constantly one: the continuity of the source term does not help to this extent}.

\item[\S~\ref{Ss:nondifferentiabilityset}]
We construct a convex flux $f$ and a continuous Eulerian solution $u$, which is then Lipschitz continuous along all characteristic curves \cite[Theorem 30]{file1ABC}. Surprisingly, at points of a compact set $K \subset \mathbb{R}^{2}$ with positive $\mathcal{L}^{2}$-measure, $u$ is not pointwise differentiable along characteristic curves: this phenomenon occurs for all characteristic curves passing through such points.
On $K$, there is no candidate value for the derivative of $u$ along characteristic curves.
The derivative of $u$ along characteristic curves does not provide a function $\gfr$ on the plane, not even almost everywhere: negligible along characteristics does not imply negligible in the plane.
We emphasize that such function $u$ is not Hölder continuous and that this behavior is prevented by the $\alpha$-convexity of the flux~\cite[Theorem~1.2]{estratto}—our flux is only strictly convex.

\item[\S~\ref{S:Cantoparam}] We construct an $\Ll^2$-positive measure set $K$ negligible along characteristic curves also in the case of the quadratic flux $f(u)=u^2$.
Even in this scenario, the set $K$ intersects any characteristic curve of $u$ in at most one point.

{{}Besides showing that, surprisingly, even for the quadratic flux Lagrangian parameterizations can really have a Cantor part, this construction introduces the general machinery used in the more complex counterexample of \S~\ref{Ss:nondifferentiabilityset}.
It also shows that Lagrangian sources might not be Eulerian sources.}

\item[\S~\ref{S:compatibilitysources}]
Despite the counterexamples we described above, comfortingly, when inflection points of $f$ are negligible source terms in the Eulerian and Broad formulation are compatible.
This is the case of analytic fluxes: there exists a Borel, bounded function $\gfr$ that is both the source term for~\eqref{EE} and for~\eqref{LE}, whichever characteristic curve $i_{\gamma}$ one chooses.
We stress that, even with the cubic flux $f(u)=u^3$, due to the non-convexity, \emph{if the Eulerian source is continuous the continuous representative is not necessarily the right Lagrangian source}, see a counterexample in Remark~\ref{Rem:cubico}. 
If the Lagrangian source is a continuous function, instead, then it is the right Eulerian source.
\end{itemize}

We first collect and summarize in~\S~\ref{S:2} precise definitions and statements. 

\begin{remark}
As we discuss local properties, for simplicity of notation we assume that $u$ is defined on all $\R^2$. This is mostly a notational convenience, and it would be entirely similar when $u$ is defined in an open set.
\end{remark}


\section{Description of the setting and of the statements}
\label{S:2}

Let $u:\R^2\to\R$ be a given \emph{continuous} function and $ f:\R\to\R$ a given function of class $C^{2}$, that we call flux function.

We call $z^{*}\in\R$ an inf{}lection point of $f$ if $f''(z^{*})=0$ but $z^{*}$ is neither a local maximum nor a local minimum for $ f(z)-f'(z^{*})(z-z^{*})$.
We denoted by $\infl(f)$ the set of inf{}lection points of $f$, $\clos({\infl(f)})$ is its closure.

By negligibility of inflection points of $f$ we mean the following assumption:
\begin{equation}
\Ll^1(\clos(\infl(f)))=0
\end{equation}
where, with definitions that can be proved to be equivalent, we denote
\begin{equation}
\clos(\infl(f))\equiv\acc( \{f''>0\})\bigcap\acc( \{f''<0\})
\label{ass:h}\,.
\end{equation}
{{}Above $\acc(X)$ denotes the set of accumulation points of a set $X$.}
\nomenclature{$\infl(f)$}{Inflection points of $f$, see Assumption~\eqref{ass:h}}
\nomenclature{$\clos(X)$}{Closure of $X$}
\nomenclature{$\acc(X)$}{{{}Accumulation points of $X$}}
\nomenclature{$\ri(X)$}{{{}The relative interior of a set $X$}}

We review in which sense $u$ can be a \emph{continuous} solution of the balance law 
\[
\pt u + \px (f(u)) = \gfr
\] 
for a Borel, bounded function $\gfr$.
Depending on assumptions on the flux and on the function $\gfr$, $u$ can satisfy all the interpretations we consider of such equation, or not.

\subsection{Review of different interpretations}
Define \emph{characteristic curves} as $C^1$ integral curves $i_{\gamma}{{}:I\to\R^{2}}$ of the vector field $(1,f'(u))$, where {{}where $I$ is an interval}.
Define characteristics as $C^{1}$ solutions $\gamma{{}:I\to\R}$ of $\dot\gamma(t)=f'(u(t,\gamma(t)))$, where {{}where $I$ is an interval}.
\nomenclature{$(t,\gamma(t))$}{Characteristics of~\eqref{EE}: $\gamma\in C^1(\R;\R^2)$ solves $\dot\gamma(t)=f'(u(t,\gamma(t)))$}
In the following definition we collect a monotone family of such characteristics and we define a change of variables in $\R^2$.
\begin{definition}
\label{D:LagrangianParameterization}
We call \emph{(full) Lagrangian parameterization}, associated with the continuous function $u$, a continuous function $\chi:\R^{2}\to\R$ such that\footnote{If the first two conditions are required $\Ll^{1}$-a.e., then they hold naturally for all parameters.}
\begin{itemize}
\item[-] for each $y\in\R$, the function $t\mapsto\chi(t,y)=\chi_{y}^{\reu}(t)$ is a characteristic:
\[
\dot \chi_{y}^{\reu}(t)=\pt\chi(t,y)=f'(u(t,\chi(t,y)))=f'(u(i_{\chi(y)}(t)));
\]
\item[-] for each $t\in\R^{+}$, $y\mapsto\chi(t,y)=\chi^{\red}_{t}(y)$ is nondecreasing;
\item[-] denoting $i_{\chi}(t,y)\equiv i_{\chi(y)}(t)\equiv(t,\chi(t,y))$, then $i_{\chi}$ is onto $\R^{2}$.
\end{itemize}
\nomenclature{$i_{\chi}$}{$i_{\chi}(t,y)\equiv i_{\chi(y)}(t)\equiv(t,\chi(t,y))$, see Definition~\ref{D:LagrangianParameterization}}
\nomenclature{$\chi$}{Lagrangian parameterization for a continuous solution $u$ to~\eqref{EE}, see Definition~\ref{D:LagrangianParameterization}}
\nomenclature{$\D(\Omega)$}{{Distributions on the open set $\Omega$}}
\nomenclature{Lagrangian parameterization}{\hfill See Definition~\ref{D:LagrangianParameterization}}
A Lagrangian parameterization $\chi$ is \emph{absolutely continuous} if $(i_{\chi}^{-1})_{\sharp}\Ll^{2}\ll \Ll^{2}$.
Equivalently, $\chi^{-1}(S)$ must have positive $\Ll^{2}$-measure if $\Ll^2(S)>0$: $\chi$ maps negligible sets into negligible sets.
We say that a given Borel function $\gfr$ is a \emph{Lagrangian source} associated to $u$ and to the Lagrangian parameterizaion $\chi$ if
\begin{equation}
\forall y\in\R\qquad
\ddt u(t,\chi(t,y))=\gfr(t,\chi(t,y))
\qquad \text{in $\D'(\R^{+})$.}
\end{equation}
\end{definition}

Let $G>0$, $u:\R^{2}\to\R$ be continuous and $f:\R\to\R$ of class $C^{2}$. The following conditions are equivalent~\cite[Lemma 45--Corollary 46--Corollary 28]{file1ABC}:
\begin{enumerate}\label{itemize:sols}
\item \textbf{Eulerian solution}: The equation $\pt u(t,x) + \px (f(u(t,x))) =\gfr(t,x)$ holds in distributional sense for a Borel function $\gfr$ bounded by $G$.
\item \textbf{Kruzkov iso-entropy solution}: For every $\eta,q\in C^{1}(\R)$ satisfying $q'=\eta'f'$  \[\pt (\eta(u(t,x))) + \px (q(u(t,x))) =\eta'(u(t,x))\gfr(t,x)\] holds in distributional sense for a Borel function $\gfr$ bounded by $G$.
\item $u$ is Lipschitz continuous with constant $G$ along a family of characteristic curves whose image is dense in $\R^2$.
\item \textbf{Lagrangian solution}: $ t\mapsto u(t,\chi(t,y))$ is Lipschitz continuous with constant $G$ for a Lagrangian parameterization $\chi$ as in Definition~\ref{D:LagrangianParameterization} and for each $y$.
\end{enumerate}
If $\Ll^2(\clos(\infl(f)))=0$, the conditions above are also equivalent to \cite[\S~3]{file1ABC}:
\begin{enumerate}\setcounter{enumi}{3} 
\item \textbf{Broad solution}: $u$ is $G$-Lipschitz continuous along \emph{all} characteristic curves.
\end{enumerate}
\nomenclature{Eulerian solution}{\hfill See Page~\pageref{itemize:sols}}
\nomenclature{Broad solution}{\hfill See Page~\pageref{itemize:sols}}
\nomenclature{Lagrangian solution}{\hfill See Page~\pageref{itemize:sols}}
Summarizing, we established in~\cite{file1ABC} the equivalences
{

\centering\begin{tabular}{ccccccccc}
Broad 
&& \begin{tabular}{c} $\Longrightarrow$ always\\ \hline $\Longleftarrow$ if~\eqref{ass:h} holds \end{tabular}
&& Lagrangian
&&$\Longleftrightarrow$
&& Eulerian \\
\end{tabular}
\\ \noindent
}without discussing the identification of sources, and we proved entropy conservation.

\smallskip

We also established the following relations among sources in the different formulations, that we picture in Figure~\ref{fig:sources}, when inflections of the flux $f$ are negligible.

\begin{definition}
Lagrangian sources are explained in Definition~\ref{D:LagrangianParameterization}. Eulerian sources are Borel functions $g$ for which~\eqref{EE} holds in $\D'(\R^{2})$. Broad sources are Borel functions $g$ satisfying~\eqref{LE} \emph{for every} characteristic curve $i_{\gamma}$.
\end{definition}

\begin{theorem}[{\cite[Lemma 16-Theorem 37-Corollary 46]{file1ABC}}]
Let $\chi$ be any Lagrangian parameterization.
Then the family of sources associated to the Lagrangian parameterization $\chi$ contains the family of Broad sources.

If there exists an Eulerian source, then:
\begin{itemize}
\item The family of Lagrangian sources is non-empty.
\item If Assumption~\eqref{ass:h} holds, then the family of Broad sources is also nonempty.
\end{itemize}
\end{theorem}
We emphasize that any Broad source $\gfr$ is a good Lagrangian source independently of the choice of the Lagrangian parameterization, which always exists \cite[Lemma~17]{file1ABC}.

\smallskip

\subsection{Identification of sources}

We prove in \S~\ref{S:compatibilitysources} the following positive statement, see Theorem~\ref{T:nonemptyint} below, which includes the case of analytic functions. We emphasize that in the case of a continuous Eulerian source the continuous representative of the source is not necessarily, as one would expect, the right Broad source: we refer to Remark~\ref{Rem:cubico} for a counterexample with $f(u)=u^3$, due to non-convexity, and Theorem~\ref{T:continousSourceGen} for a weaker positive result.
A Lagrangian continuous source instead is automatically an Eulerian continuous source, actually even without Assumption~\eqref{ass:h}.

\begin{theorem}
If inflection points of $f$ are negligible as specified in~\eqref{ass:h}, then the family of Eulerian sources has nonempty intersection with the family of Broad sources.
\end{theorem}

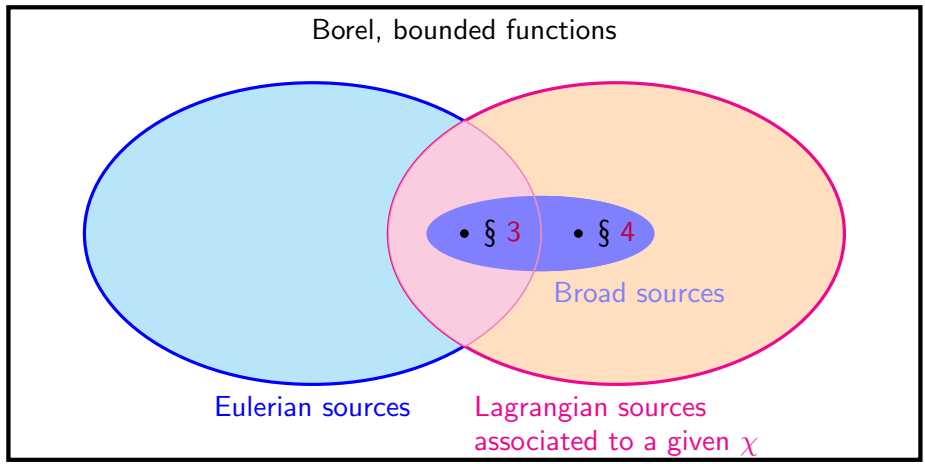
\begin{figure}[ht!]
\begin{tikzpicture}
\draw[ultra thick, ->] (0,0) rectangle (12,6) 
(6,6) node[align=left, below] {Borel, bounded functions};
\filldraw[color=blue, fill=cyan!25, very thick](4,3) circle (3 and 2) (4,1) node[align=left, below] {Eulerian sources};
\filldraw[color=magenta!99, fill=orange!25, very thick](8,3) circle (3 and 2) (8,1) node[align=left, below] {Lagrangian sources\\ associated to a given $\chi$};
\fill[blue!50] (7,3) ellipse (1.5 and 0.5) (8.3,2.5) node[align=left, below] {Broad sources};
\fill[blue!50] (7,3) ellipse (1.5 and 0.5);
\begin{scope}
  \clip (8,3) circle (3 and 2);
  \filldraw[color=magenta!60, fill=magenta!25, very thick](4,3) circle (3 and 2);
\end{scope}
\begin{scope}
  \clip (4,3) circle (3 and 2);
  \filldraw[color=magenta!60, fill=magenta!25](8,3) circle (3 and 2);
\end{scope}
\begin{scope}
  \clip (4,3) circle (3 and 2);
  \fill[blue!50] (7,3) ellipse (1.5 and 0.5);
\end{scope}
\filldraw (6,3) circle (.05) (6.5,3) node[align=right] {\S~\ref{S:compatibilitysources}};
\filldraw (7.5,3) circle (.05) (8,3) node[align=right] {\S~\ref{S:Cantoparam}};
\end{tikzpicture}
\caption{We picture relations among the sources that we determine for a fixed continuous solution of the balance law~\eqref{EE} under the non-degeneracy Assumption~\eqref{ass:h}.
{{}When a Lagrangian source is continuous, it is also an Eulerian source, although, surprisingly, it might not be the Lagrangian source for a different Lagrangian parameterization, see Remark~\ref{Rem:cubico} and Theorem~\ref{T:continousSourceGen}. If a broad source is continuous, it is a right source with all the formulations}.}
\label{fig:sources}
\end{figure}

\subsection{Counterexamples}
Counterexamples are the most surprising part of this work, pointing out differences among different formulations. Of course the family of Eulerian sources is not generally contained in the family of Lagrangian sources, since changing values on a curve affects both the Lagrangian and the Broad formulation but not the Eulerian one. Other relations in Figure~\ref{fig:sources} are less trivial.

\vspace{.3\baselineskip}
\texttt{Section~\ref{S:Cantoparam}}. The first counterexample concerns Lagrangian parameterizations, but not only.
We already mentioned that one can build up a monotone family of characteristics in order to define a `monotone' change of variables $i_\chi$ from $\R^2$ to $\R^2$.
We know from~\cite[Example A.2]{BCSC} that in general one cannot choose a Lagrangian parameterization $\chi$ which is Locally Lipschitz continuous. In \S~\ref{S:Cantoparam} we construct an example---for the quadratic flux!!!---where the measure $\partial_y \chi$ has a Cantor part.
In particular, the Lagrangian parameterization is not absolutely continuous, according to Definition~\ref{D:LagrangianParameterization}. Even for the quadratic flux thus, surprisingly, there exists a subset $K$ of the plane of positive $\Ll^2$-measure which intersects each characteristic curve in a single point.
For this reason, even for the quadratic flux, the family of Broad sources is not contained in the family of Eulerian sources: the Lagrangian source can be defined arbitrarily on the $\Ll^2$-positive measure set $K$, but of course not all such definitions provide an Eulerian source.

\vspace{.3\baselineskip}
\texttt{Section~\ref{Ss:nondifferentiabilityset}}.
For uniformly convex fluxes, there is a natural element lying both in the family of Eulerian sources and in the family of Lagrangian sources (\cite[Corollary~6.7]{BCSC}, \cite[Theorem~1.2]{estratto}, \cite{Pinamonti}): the pointwise derivative of $u$ along any characteristic curve through the point, when possible, and vanishing in the remaining $\Ll^2$-negligible set.
We construct in this paper a second example which shows that no such natural element exists in general even for convex fluxes, not uniformly convex: there can be a positive $\Ll^2$-measure set $K\subset \R^2$ of points where $u$ is not differentiable along characteristic curves.

\vspace{.3\baselineskip}
\texttt{Section~\ref{Ss:Notlipanlongchar}}.
In the third place, we show that, when Assumption~\eqref{ass:h} fails, a continuous function $u$ which is both an Eulerian and a Lagrangian solution is not necessarily a Broad solution: the family of Broad sources might be empty, if inflections of $f$ are not negligible. Namely, Example~\ref{Ex:Lagrnobroad} provides a continuous solution  $u(t,x)\equiv u(x)$ to a balance law which is \emph{not} Lipschitz continuous on many of its characteristics.
It also proves that, whenever one chooses a `good' Lagrangian parameterization, the Lagrangian source should be fixed accordingly: no universal choice is possible even within all admissible Lagrangian parameterizations. 

\vspace{.3\baselineskip}
\texttt{Section~\ref{S:cotinuoussources}}.
{{}
With continuous sources, the set of broad solutions might be always non-empty. 
Even if this was the case, Example~\ref{Ex:cubico} with $f(u)={u^{2}}$, $u(t,x)=\sqrt[3]x$ and $g(t,x)=1$ shows that the Broad source, that in this case is available, is not the continuous representative of the continuous, constant Eulerian source.
Nevertheless when a Lagrangian source is continuous then it is also the correct Eulerian source.
}

\section{Compatibility of Broad and Eulerian sources when inflections are negligible}
\label{S:compatibilitysources}

When the flux has negligible inflection points, the source terms in the Broad, Lagrangian and Eulerian interpretations of~\eqref{EE} are compatible.

\begin{theorem}
Let $u$ be continuous. Suppose the distribution $ \pt u(t,x) + \px (f(u(t,x))) $ is represented by a bounded function. 
If $\Ll^{1}(\clos({\infl(f)}))=0$, then the family of Broad sources and the family of Eulerian sources have nonempty intersection.
\end{theorem}

We prove such positive statement in this section by constructing a Borel function which is both an Eulerian source and a Broad source: see Theorem~\ref{T:nonemptyint}.

In~\cite[Definition 36, Theorem 37]{file1ABC} we constructed the Broad source
\begin{equation}\label{E:broadSource}
\gfr_{B}(t,x)
\doteq 
\begin{cases}
\gfr (t,x)
&(t,x)\in E\setminus u^{-1}(\clos({\infl(f)}))\\
0
&\text{otherwishe}
\end{cases}
\end{equation}
where $E$ is the Borel set \cite[equation (3.4) and Lemma 38]{file1ABC} of points $(t,x)$ through which there exists a $C^1$ (time-translated) characteristic $\gamma $, where $\gamma (0)=(t,x)$, for which $s\mapsto u(t +s, \gamma (s))$ is differentiable at $s=0$; the function
$(t, x)  \mapsto \gfr(t, x) =\dds u(t + s, \gamma(s)) \Big|_{s=0}
$ was defined by a selection theorem on such set $E$.

In order to have that the source is also an Eulerian source, of course we have to modify $\gfr_{B}$ outside $E{{}\cup u^{-1}(\clos({\infl(f)}))}$ setting it equal to any Eulerian source $\gfr_E$ instead of fixing a `random' value {{}and we need to prove that it is the correct Eulerian source on  $E\cup u^{-1}(\clos({\infl(f)}))$}. Changing the value outside $E$ does not affect the fact of being a Broad source term, because $u$ is Lipschitz continuous along any characteristic curve \cite[Theorem 30]{file1ABC} and therefore the complement of $E$ is $\Haus^1$-negligible along any characteristic curve.
Nevertheless, we could also need to turn $\gfr_{B}$ into $\gfr_E$ within $E$: we indeed construct a subset $ {} D_{\gfr_E}$ of $\{\gfr_B\neq\gfr_E\}$ of its same (maybe positive!) $\Ll^2$-measure and which is negligible along any integral curve of $(1,f'(u))$.
Once we turn $\gfr_{B}$ into $\gfr_E$ also on ${}  D_{\gfr_E}$ we are done and we get a Borel function which is both a Broad and an Eulerian source term.

We first prove that $0$ is a good value both for Eulerian and Broad sources on $u^{-1}(\clos({\infl(f)}))$, thanks to the assumption of negligibility of inflection points.

\begin{lemma}
\label{L:setNull}
Consider any closed set $N\subset\R$ which is $\Ll^1$-negligible. Then
\begin{enumerate}
\item\label{item:posiisiss} any Eulerian source $\gfr_E$ vanishes at $\Ll^2$-Lebesgue points of $u^{-1}(N)\subseteq\R^2$.
\item\label{item:possssiisiss} the $t$-derivative of $u\circ i_{\gamma }(t)=u(t,\gamma(t))$ vanishes at $\Haus^1$-Lebesgue points of $(u\circ i_{\gamma })^{-1}(N)\subseteq\R$, for any any $C^1$ integral curve $i_{\gamma}$ of $(1,f'(u))$
\end{enumerate}
\end{lemma}

\begin{proof}
By $\sigma$-additivity, we directly assume that $N$ is compact, not only closed.\\
\eqref{item:posiisiss}
We apply the entropy equality~\cite[Lemma~42]{file1ABC}.
Choose in particular the convex entropies
\[
\eta'_{\varepsilon}(z)=(\uno_{O_{\varepsilon}}*\rho_{\varepsilon})(z),
\qquad \eta_{\varepsilon}(-\infty)=0,
\]
where $\rho_{\varepsilon}$ is a smooth convolution kernel concentrated on $[-\varepsilon,\varepsilon]$ and $O_{\varepsilon}\supset N$ is a sequence of open sets such that $\Ll^{1}(O_{\varepsilon})<\varepsilon$.
Since $\Ll^{1}(N)=0$ and $N$ is compact, then $\eta_{\varepsilon}'(u)$ converges pointwise to $\mathds{1}_{u^{-1}(N)}$ so that $\eta_{\varepsilon}(z)$ converges locally uniformly to $0$.
In particular, the distributions $\pt\eta_{\varepsilon}(u)$ and $\px (q_{\varepsilon}(u)) $ must vanish in the limit as $\varepsilon\downarrow0$.
From the entropy equality
\[
\pt\eta_{\varepsilon}(u) + \px (q_{\varepsilon}(u)) 
=
\eta_{\varepsilon}'(u)\gfr_{E} 
\qquad\text{in $\D'(\R^2)$}
\]
we deduce the claim in the limit $\varepsilon\downarrow0$: $\gfr_{E}$ must vanish $\Ll^{2}$-a.e.~on $u^{-1}(N)$ because
\[
0=\mathds{1}_{u^{-1}(N)}(t,x)\gfr_E(t,x)
\qquad\text{in $\D'(\R^2)$.}
\]

\eqref{item:possssiisiss} Recall that $u$ is $G$-Lipschitz continuous along characteristic curves \cite[Theorem~30]{file1ABC}, thus in particular $u\circ i_{\gamma}$ is $G$-Lipschitz. As $N$ is $\Ll^{1}$-negligible by hypothesis then the derivative of $u\circ i_{\gamma}(t)$ must vanish at Lebesgue points of $(u\circ i_{\gamma})^{-1}(N)$ by an easy computation, see~\cite[Lemma~41]{file1ABC}.
\end{proof}

{{}We now construct in Lemma~\ref{L:vario} a subset $  D_{\gfr_E}$ of $E$, defined in~\eqref{finalset}, where we need to turn $\gfr_B$ into $\gfr_E$: in Lemma~\ref{L:setdifferent} we remove from $E$ 
\begin{itemize}
\item the set where $\gfr_E=\gfr_B$, 
\item the set $u^{-1}(\clos({\infl(f)}))$ and 
\item the points which are not Lebesgue points of the time-restrictions of $\gfr_E$.
\end{itemize}
We prove that the remaining set $D_{\gfr_E} $ is negligible along any Lagrangian parameterization.
Corollary~\ref{C:varioC} then proves that $D_{\gfr_E}$ is negligible not only along any Lagrangian parameterization, but also along any integral curve of $(1,f'(u))$.
Once constructed $   D_{\gfr_E}$, the compatible source in Theorem~\ref{T:nonemptyint} will be straightforward.
}

\begin{lemma}
\label{L:setdifferent}
Let $S^\complement$ denote the complement of a set $S$. Let $\gfr_{B}$ be the Broad source term in~\eqref{E:broadSource} and consider any Eulerian source term $\gfr_E$. Define 
\begin{align}\label{finalset}
D_{\gfr_E}\doteq E&\bigcap \{\gfr_E=\gfr_B\}^{\complement}  
	\bigcap\left( u^{-1}\left(\clos({\infl(f)})\right)\right)^\complement\\
	&	\bigcap\left\{\exists\lim_{h\to0}\frac{1}{h}\int_{\bar x}^{\bar x+h} \gfr_{E}(\bar t,x)dx= \gfr_E(\bar t,\bar x)\right\} \ .\notag
\end{align}
Then $\Ll^2(\Psi^{-1}(D_{\gfr_E}))=0$ for any Lagrangian parameterization $\chi$, where \[\Psi(t,y)=(t,\chi(t,y))\,.\]
\end{lemma}
\begin{proof}
Notice that we are not considering the intersection of the accumulation points of $\{f''>0\}$ and $\{f''<0\}$ since such intersection is precisely $\clos({\infl(f)})$, which lies in the complement of $D_{\gfr_E}$.
We can thus focus in an open set where $\{f''\geq 0\}$: open sets where $\{f''\leq 0\}$ are analogous.

{{}Denote by $\ri(X)$ the relative interior of a set $X$.}
Fix any Lagrangian parameterization $\chi$ and set $\Psi(t,y)=(t,\chi(t,y))$: we prove that \[\Ll^2\left(\Psi^{-1}\left(D_{\gfr_E} \right)\right)=0 .\]

Let $(\bar t, \bar x)\in D_{\gfr_E}\cap \ri\{f''(u)\geq 0\}  $.
Suppose $(\bar t, \bar x)=(\bar t, \chi(\bar t,\bar y))=\Psi(\bar t,\bar y)$ with
\begin{align}
\label{E:sionwgnuigwn2222}
&\exists \lim_{\varepsilon\to 0}\frac{1}{\varepsilon}\int_{\bar t}^{\bar t+\varepsilon}\Gfr_{E}(  t,\bar y)dt=\gfr_{E}(\bar t, \bar x)
\qquad\text{where }\Gfr_{E}(  t,  y)= \gfr_{E}( t,\chi(  t,  y))
\\
\label{E:sionwgnuigwn}
&\exists \lim_{\varepsilon\downarrow 0}\frac{\Ll^1\left(\Psi^{-1}(D_{\gfr_E})\cap[\bar t-\varepsilon,\bar t+\varepsilon]\times\{\bar y\}\right)}{2\varepsilon}=1
\\
\label{E:sionwgnuigwn111}
&\exists \lim_{\varepsilon\to 0}\frac{u(\Psi(\bar t+\varepsilon,\bar y))-u(\Psi(\bar t,\bar y))}{\varepsilon}=\gfr_{B}(\bar t,\bar x)
\end{align}
Since this is satisfied at $\Ll^2$-a.e.~$(t,y)$ in $\Psi^{-1}(D_{\gfr_E})$, the thesis will not be affected.

One can go back to Dafermos' computation~\cite[(3.1a)]{file1ABC}, which means in the integral formulation of the PDE, and exploit first the sign information on $f''$: what we get is 
\begin{equation}
\label{E:ruibewgde}
\frac{1}{\varepsilon (\tau-\sigma)}\left\{\int_{\gamma(\tau)}^{\gamma(\tau)+\varepsilon} u(\tau,x) dx
-\int_{\gamma(\sigma)}^{\gamma(\sigma)+\varepsilon} u( \sigma,x)dx
-\int_{\sigma}^{\tau}\int_{\gamma(t)}^{\gamma(t)+\varepsilon}\gfr_{E}( t,x)dxdt\right\}\leq 0
\end{equation}
whenever $\{\tau\leq t\leq\sigma\ ,\ \gamma(t)\leq x\leq \gamma(t)+\varepsilon\}$ is contained in $\{f''\geq0\}$.
We apply this integral relation fixing the characteristic $\gamma(t)=\chi(t,\bar y)$ and $\sigma\equiv \bar t$.
We show below that by the choice of $(\bar t, \bar x)$ one can pass to the limit, first as $\varepsilon\downarrow0$, then as $\tau\downarrow \sigma\equiv \bar t$.

\firststep
\step{Last addend}
When $(t,\gamma(t))\in D_{\gfr_E} $, the space average converges by definition of $D_{\gfr_E}$:
\[
\exists \lim_{\varepsilon\downarrow 0}\frac{1}{\varepsilon}\int_{\gamma(t)}^{\gamma(t)+\varepsilon} \gfr_{E}( t,x)dx=\gfr_{E}(t,\gamma(t))
\]
In particular, when $(t,\gamma(t))\in D_{\gfr_E} $
\[
\gfr_{E}(t,\gamma(t))
= \liminf_{\varepsilon\downarrow 0}\frac{1}{\varepsilon}\int_{\gamma(t)}^{\gamma(t)+\varepsilon} \gfr_{E}( t,x)dx=
 \limsup_{\varepsilon\downarrow 0}\frac{1}{\varepsilon}\int_{\gamma(t)}^{\gamma(t)+\varepsilon} \gfr_{E}( t,x)dx
\]
When $(t,\gamma(t))\notin D_{\gfr_E}$ perhaps the limit does not exists, but the liminf and the limsup are bounded by $\pm G$.
Thanks to this bound and  by~\eqref{E:sionwgnuigwn}, the time-average is the same when the average is done considering only the points $t$ with $(t,\gamma(t))\in D_{\gfr_E}$: denoting 
\[
\Psi^{-1}(D_{\gfr_E})\cap[\bar t-\varepsilon,\bar t+\varepsilon]\times\{\bar y\}=:(\Psi^{-1}(D_{\gfr_E}))^{\reu}_{\bar y}
\]
by~\eqref{E:sionwgnuigwn2222} we thus deduce
\begin{align*}
\gfr_{E}(\bar t,\bar x) 
&=\lim_{\tau\downarrow \sigma}\frac{1}{|\tau-\sigma|}\int_{[\sigma,\tau]\cap(\Psi^{-1}(D_{\gfr_E}))^{\reu}_{\bar y}}\gfr_{E}( t,\gamma(t))dt
\\
&=\lim_{\tau\downarrow \sigma}\frac{1}{|\tau-\sigma|}\int_{[\sigma,\tau]}\liminf_{\varepsilon\downarrow 0}\frac{1}{\varepsilon}\int_{\gamma(t)}^{\gamma(t)+\varepsilon} \gfr_{E}( t,x)dxdt
\\
&=\lim_{\tau\downarrow \sigma}\frac{1}{|\tau-\sigma|}\int_{[\sigma,\tau]}\limsup_{\varepsilon\downarrow 0}\frac{1}{\varepsilon}\int_{\gamma(t)}^{\gamma(t)+\varepsilon} \gfr_{E}( t,x)dxdt\ .
\end{align*}
We conclude by Fatou's lemma, which proves the double limit for the last addend because it implies
\begin{align*}
&\lim_{\tau\downarrow \sigma}\liminf_{\varepsilon\downarrow 0}\frac{1}{|\tau-\sigma|\varepsilon}\int_{\sigma}^{\tau}\int_{\gamma(t)}^{\gamma(t)+\varepsilon} g( t,x)dxdt
\geq \gfr_{E}(\bar t,\bar x) 
  \\
&\gfr_{E}(\bar t,\bar x) \geq  \lim_{\tau\downarrow \sigma}\limsup_{\varepsilon\downarrow 0}\frac{1}{|\tau-\sigma|\varepsilon}\int_{\sigma}^{\tau}\int_{\gamma(t)}^{\gamma(t)+\varepsilon} g( t,x)dxdt.
\end{align*}
\step{First addends} Recall that $u$ is differentiable at $(\bar t, \bar x)$ along $\chi(t,\bar y)$ by~\eqref{E:sionwgnuigwn111} with derivative $\gfr_{B}(\bar t, \bar x)$. By the continuity of $u$ then one has that the first two addends in the LHS converge to $\gfr_{B}(\bar t, \bar x)$:
\begin{align*}
\lim_{\tau\downarrow \sigma}&\lim_{\varepsilon\downarrow 0}\frac{1}{|\tau-\sigma|\varepsilon}\left[\int_{\gamma(\tau)}^{\gamma(\tau)+\varepsilon} u(\tau,x) dx
-\int_{\gamma(\sigma)}^{\gamma(\sigma)+\varepsilon} u( \sigma,x)dx\right]
\\
=&\lim_{\tau\downarrow \sigma}\frac{u(\tau,\gamma(\tau))  
-  u( \sigma,\gamma(\sigma))}{|\tau-\sigma|}
 = \gfr_{B}(\bar t, \bar x).
\end{align*}
\step{Conclusion}
By the double limits proved in the previous sub steps, the inequality~\eqref{E:ruibewgde} yields 
\[
\gfr_{B}(\bar t,\bar x) -\gfr_{E}(\bar t, \bar x) \leq 0.
\]
In particular, $\gfr_{B}(\bar t,\bar x) \leq\gfr_{E}(\bar t, \bar x) $. The reverse inequality comes considering the similar region bounded by $\gamma(t)$ and $\gamma(t)-\varepsilon$ instead of $\gamma(t)$ and $\gamma(t)+\varepsilon$.
As the set $D_{\gfr_E}\cap \ri\{f''\leq 0\}$ is entirely analogous, we conclude thus $\gfr_{B}(\bar t,\bar x) =\gfr_{E}(\bar t, \bar x)$ at those points $\Psi(\bar t,\bar y)=(\bar t, \bar x)\in D_{\gfr_E}$ satisfying~\eqref{E:sionwgnuigwn2222}-\eqref{E:sionwgnuigwn}-\eqref{E:sionwgnuigwn111}. Taking into account that $D_{\gfr_E}\subset \gfr_{B}(\bar t,\bar x) =\gfr_{E}$ by its definition in Lemma~\ref{L:setdifferent}, $D_{\gfr_E}$ is only made of points satisfying~\eqref{E:sionwgnuigwn2222}-\eqref{E:sionwgnuigwn}-\eqref{E:sionwgnuigwn111} and thus $\Ll^2(\Psi^{-1}(D_{\gfr_E}))=0$.
\end{proof}


\begin{lemma}{}\label{L:trivialChi}
Given any characteristic curve $\overline\gamma$, there is a Lagrangian parameterization $\chi$ with $\chi(\cdot,\overline y)=\overline\gamma $ for some $\overline y$.
\end{lemma}

\begin{proof}{}
Consider points $\{(t_{k},x_{k})\}_{k\in\N}$ dense in the plane. 
Let $\gamma_{k}$ be a characteristic with $\gamma_{k}(t_{k})=x_{k}$, $k\in\N$.
We now modify this family of characteristics recursively in order to make it monotone and to add also the fixed characteristic $\overline\gamma$.

Set $\mathcal F_{0}=\{\overline\gamma\}$. For $k\in\N$ set $\mathcal F_{k}=\mathcal F_{k-1}\cup \{\widetilde \gamma_{k}\}$ where
\[
\widetilde \gamma_{k}(t)\doteq\max_{\gamma_{\ell}}\left\{\gamma_{\ell} (t);\min_{ \gamma^{r}} \left\{\gamma^{r}(t);  \gamma_{k} (t)\right\}\right\}
\ :\   \gamma_{\ell}, \gamma^{r}\in\mathcal F_{k-1}\ \text{with}\  \gamma_{\ell}(t_{k})\leq x_{k} \leq   \gamma^{r}(t_{k}) \,.
\]
If there is no $ \gamma^{r}\in\mathcal F_{k-1} $ satisfying $  x_{k} \leq   \gamma^{r}(t_{k})$ then the minimum above is meant to be just $\gamma_{k}(t)$, while if there is no $\gamma_{\ell} \in\mathcal F_{k-1}$ satisfying $\gamma_{\ell}(t_{k})\leq x_{k}$ then the maximum is defined to be $\min_{ \gamma^{r}} \left\{\gamma^{r}(t);  \gamma_{k} (t)\right\}$.
Notice that $\widetilde\gamma_{k}(t_{k})=x_{k}$.

Let $\mathcal F=\clos\left(\cup_{k\in\N}\mathcal F_{k}\right)$ be the closure of $\mathcal F_{k}$ in the topology of locally uniform convergence.
Since this property hods in each family $\mathcal F_{k}$, for each $\gamma_{1},\gamma_{2}\in \mathcal F$ then 
\begin{itemize}
\item either $\gamma_{1}(t)\leq \gamma_{2}(t)$ for all $t$ or  
\item $\gamma_{1}(t)\geq \gamma_{2}(t)$ for all $t$.
\end{itemize}

Similarly to what explained in the proof of~\cite[Lemma~17]{file1ABC}, the function
\[
 \mathcal F \ni\gamma\mapsto\theta(\gamma)=\sum_{k=0}^{+\infty} \frac{\tanh\left(\gamma(t_{k})\right)}{2^{k}}\in(-2,2)\]
is continuous and strictly order preserving: defining $\theta^{-1}$ its continuous inverse,
\[
\chi(t,y)=\left[\theta^{-1}(y)\right](t)\qquad
y\in \theta\left( \mathcal F\right)
\]
provides a Lagrangian parameterization.

Since by construction $\overline \gamma$ belongs in each $ \mathcal F_{k}$ and thus in $\mathcal F$, just set $\overline y=\theta(\overline\gamma)$.
\end{proof}

\begin{corollary}
\label{C:varioC}{}
Let $S$ be any Borel set which is negligible along any Lagrangian parameterization: we assume that $\Ll^2(\Psi^{-1}(S))=0$ for any Lagrangian parameterization $\chi$, where $\Psi(t,y)=(t,\chi(t,y))$.
Then $S$ is negligible along any characteristic.
\end{corollary}

\begin{proof}{}
Given any characteristic curve $\overline\gamma$, there is a Lagrangian parameterization $\chi$ with $\chi(\cdot,\overline y)=\overline\gamma $ for some $\overline y$ by Lemma~\ref{L:trivialChi}.
Define then
\[
\overline\chi(t,y)=\begin{cases}
\chi(t,y) &y\leq \overline y\,,
\\
\overline \gamma(t) &\overline y\leq y\leq \overline y+1\,,
\\
\chi(t,y-1) &y\geq \overline y+1\,.
\end{cases}
\]
By assumption $S$ must be negligible also along this Lagrangian parameterization:
\[
\Ll^{2}\left(\Psi^{-1}(S)\cap\R\times[\overline y, \overline y+1]\right)=0\,,
\]
where we consider $\Psi(t,y)=(t,\overline \chi(t,y))$.
Thus we conclude by Tonelli theorem:
\[
\Haus^{1}\left( i_{\overline\gamma}(\R)\cap S \right)=
\Haus^{1}\left( i_{\overline\gamma}(\R)\cap S \right)\cdot 1
=\Ll^{2}\left(\Psi^{-1}(S)\cap\R\times[\overline y, \overline y+1]\right)=0\,.
\]
\end{proof}

\begin{lemma}
\label{L:vario}{}
The Borel subset $  D_{\gfr_E}$ of $E\cap\{\gfr_E\neq \gfr_B\}$ in~\eqref{finalset} is $\Ll^1$-negligible along any $C^1$ integral curve of $(1,f'(u))$ and it satisfies $\Ll^{2}(E\cap\{\gfr_E\neq \gfr_B\})=\Ll(  D_{\gfr_E})$.
\end{lemma}

\begin{proof}
Let us first prove that $ D_{\gfr_E}$ is itself a subset of $E\cap\{\gfr_E\neq \gfr_B\}$ with the same $\Ll^2$-measure of $E\cap\{\gfr_E\neq \gfr_B\}$.
For brevity, set $Z=u^{-1}\left(\clos({\infl(f)})\right) $.
By definition of $D_{\gfr_E}$ in Lemma~\ref{L:setdifferent} the set $E\cap  D_{\gfr_E}^\complement$ is equal to
\begin{align*}
& E\bigcap  \left(\{\gfr_E=\gfr_B\} 
\bigcup
\left(Z\setminus\{\gfr_E=\gfr_B\} \right)
 \bigcup
\left\{\exists\lim_{h\to0}\frac{1}{h}\int_{\bar x}^{\bar x+h} \gfr_{E}(\bar t,x)dx= \gfr_E(\bar t,\bar x)\right\}^\complement
\right) \ .
\end{align*}
By definition of $\gfr_B$ and by Lemma~\ref{L:setNull} then $\Ll^2\left(Z\setminus\{\gfr_E=\gfr_B\} \right)=0$. By Lebesgue differentiation theorem and by Fubini theorem, also the last set where $\gfr_E$ differs from the Lebesgue value on its $t$-sections is $\Ll^2$-negligible.
We conclude thus that
\[
\Ll^2(E\cap\{\gfr_E\neq \gfr_B\}\cap  D_{\gfr_E}^\complement)=0.
\]

The set $D_{\gfr_E}$ is negligible along any Lagrangian parameterization by Lemma~\ref{L:setdifferent} {}and it is $\Haus^{1}$-negligible along any characteristic by Corollary~\ref{C:varioC}.
\end{proof}


\begin{theorem}
\label{T:nonemptyint}
Assume that $\Ll^{1}(\clos({\infl(f)}))=0$.
Let $\gfr_E$ be any Eulerian source term.
Then the function
\[
\gfr_{U} (t,x)
\doteq 
\begin{cases}
\gfr_{B} (t,x)
&\text{on } E\setminus   D_{\gfr_E}\text{ and on }u^{-1}\left(\clos({\infl(f)})\right) \\
\gfr_{E} (t,x)
&\text{on }  D_{\gfr_E}\text{ and on $\R^2\setminus E$}
\end{cases}
\]
is both an Eulerian source term and a Broad source term.
\end{theorem}

\begin{proof}
Notice that $\gfr_E=\gfr_B$ $\Ll^2$-a.e.~on $ u^{-1}\left(\clos({\infl(f)})\right) $ by definition of $\gfr_B$ and by Lemma~\ref{L:setNull}.
Moreover $\gfr_E=\gfr_B$ $\Ll^2$-a.e.~on $E\setminus    D_{\gfr_E}$ by Lemma~\ref{L:vario}.
The function $\gfr_U$ is therefore an Eulerian source because $\gfr_U= \gfr_E$ in $\D'(\R^2)$. 

As $  D_{\gfr_E}$ is negligible along any characteristic curve by Lemma~\ref{L:vario} and $\R^2\setminus E$ is negligible along any characteristic curve by \cite[Theorem 30]{file1ABC} and by definition of $E$ \cite[equation (3.4) and Lemma 38]{file1ABC}, then $\gfr_U$ is still a Broad source, as $\gfr_{B}$ was.
\end{proof}

\section{Lagrangian parameterizations may be Cantor functions}
\label{S:Cantoparam}

{{}
We remind that for the quadratic flux $f(z)=z^{2}/2$ if $u$ has bounded variation then Lagrangian parameterizations can be taken absolutely continuous; in addition, any relative Lagrangian source term is also an admissible Eulerian source term, see~\cite[\S~2.1]{file1ABC}. 
Nevertheless, continuous solutions are necessarily $\frac12$-H\"older continuous but in general they do not have bounded variation, thus this is not the general situation.

The example in this section shows that, even with $f(z)=z^{2}/2$, one may have that any Lagrangian parameterization $\chi$ is not an absolutely continuous function, but it has a Cantor part.
In particular, whatever Lagrangian parameterization one chooses, there are Lagrangian sources relative to it that are not Eulerian sources: the Broad source shall be carefully chosen.

}

More in detail, this section aims at constructing
\begin{itemize}
\item a continuous solution $u$ of a balance law
\begin{equation}
\label{E:Burgers}
\pt u(t,x) + \px (u^{2}(t,x)) =  \gfr(t,x)
,\qquad 
|\gfr(t,x)|\leq 1\ ;
\end{equation}
\item a compact set $K\subset\R^{2}$ of positive Lebesgue measure whose intersection with \emph{any} characteristic curve of $u$ is $\Haus^{1}$-negligible.
\end{itemize}

{{}We aim at introducing here in a simpler setting the machinery that produces also the more complex counterexample of \S~\ref{Ss:nondifferentiabilityset}: we thus leave an additional parameter $d_{i}$ that in this section will turn out to be fixed constantly equal to $16$, while it will be relevant in \S~\ref{Ss:nondifferentiabilityset}.
We briefly outline the construction before presenting it:}
\begin{enumerate}
\item We partition $\R^{2}$ in a rectangle $Q_{0}$ and its complement.
We define $u=0$ on the complement of $Q_{0}$.
\item At the first step we subdivide $Q_{0}$ into finitely many sub-strips, say $d_{1}$ sub-strips equal to $S_{1}$. 
The strip $S_{1}$ is made of two sub-rectangles which are a translation of a given rectangle $Q_{1}\subset Q_{0}$ and a remaining `corridor'.
We assign a value to $u$ in each `corridor' as the derivative of a suitable family of curves $x=\gamma(t)$ covering the region of the `corridor'.
We do it in such a way that---in this closed region with $2d_{1}$ equal rectangular holes---$u$ will be a $C^{1}$ function.
By Cauchy uniqueness theorem for ODEs, all characteristic curves of $u$ in this region of the `corridor' must then coincide with the family that we assign.
\item At the $i$-th step, $i\in\N$, $u$ is defined as a $C^{1}$ function on the complement of finitely many disjoint equal rectangles which are translations of a given rectangle $Q_{i}\subset Q_{i-1}$.
We subdivide $Q_{i}$ recursively into finitely many sub-strips, say $d_{i+1}$ sub-strips. Each strip is made by two rectangles translation of $Q_{i+1}\subset Q_{i}$ and a remaining `corridor'. We assign a value to $u$ on each `corridor' so that this extension of $u$ becomes a $C^{1}$-function on the closed region with $2^{i}d_{1}\cdot\dots \cdot d_{i}$ equal rectangular holes which are a translation of $Q_{i+1}$.
\item By the previous steps we will have assigned a value to $u$ on the whole complement of a compact Cantor-like set $K\subset \R^{2}$ of positive Lebesgue measure but with empty interior: we will be able to assign a unique value of $u$ on $K$ extending $u$ to $K$ by continuity.
\item By the details of our construction, \emph{every} characteristic curve of $u$ will intersect $ K$ in at most one point.
We obtain this property by requiring that \emph{every} characteristic curve must intersect at most \emph{one} of the disjoint translation of $Q_{i}$ considered at the $i$-th step. This is possible because we impose $u\in C^{1}$ on compact subsets of the open set $\R^{2}\setminus K$, therefore characteristic curves of $u$ are uniquely defined in the open set $\R^{2}\setminus K$: we force that every characteristic curve reaching the boundary of a translation of $Q_{i}$ does not intersect any other translation of $Q_{i}$.
\end{enumerate}
After performing this program the counterexample will be ready: the counter-image of the $\Ll^{2}$-non-negligible set $K$ by any Lagrangian parameterization must have null measure, because each vertical section of this counterimage is made by the single point of intersection with $K$.
This shows that there is no Lagrangian parameterization satisfying the absolute continuity of Definition~\ref{D:LagrangianParameterization}

\begin{figure}[htp!]
\input{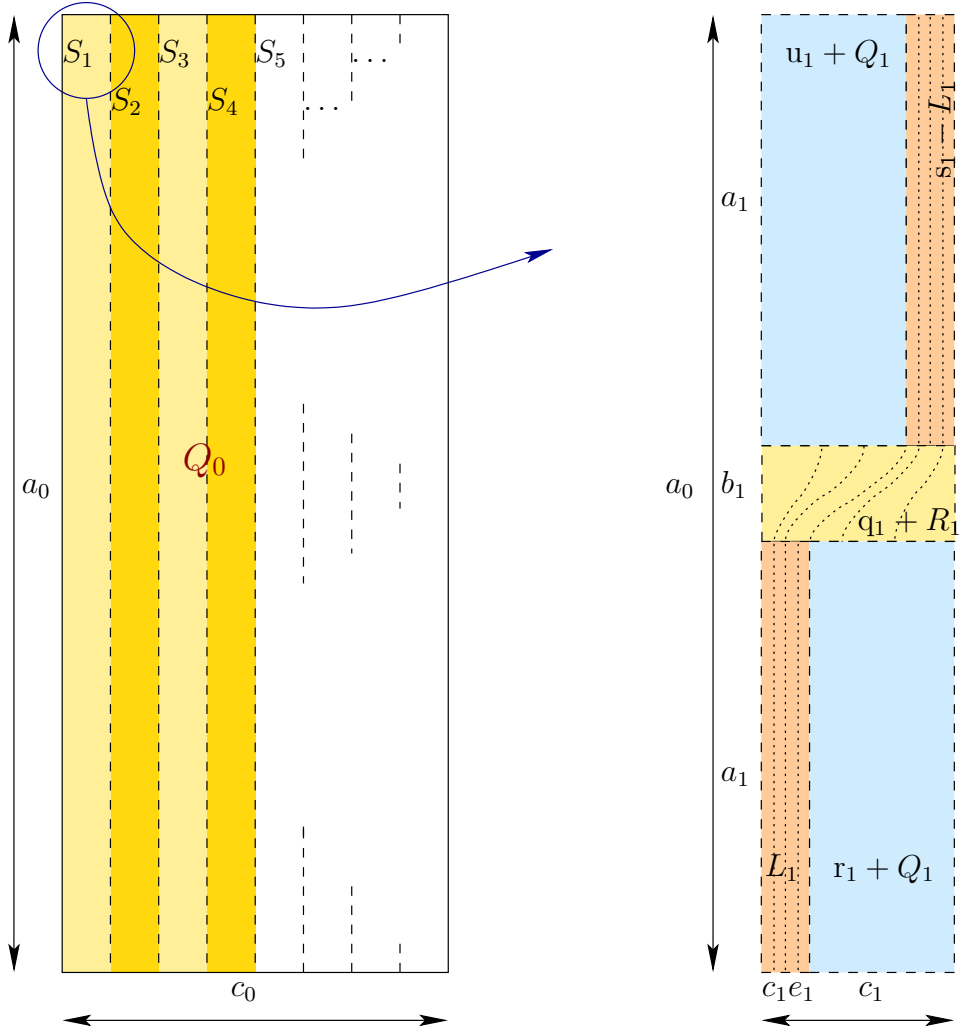}
\caption{The initial region $Q_{0}$ and one of its strips $S_{1}$. Proportions are distorted. Dashed lines suggest the qualitative behavior of characteristic curves.}
\label{fig:esempioSemplice}
\end{figure}

\firststep
\step{Definition of the iterative regions (Figure~\ref{fig:esempioSemplice})}
We define here in an iterative way a finer and finer partition of a rectangle $Q_{0}$.
This construction is based on parameters $a_{i}, b_{i}, c_{i}, d_{i}, e_{i}$ that we write explicitly in the next step.
We only mention here the relation among $c_{i},d_{i}, e_{i}$
\begin{align}
\label{E:recreldec}
d_{i}c_{i}(1+e_{i})=c_{i-1} \qquad\Rightarrow\qquad c_{i}(1+e_{i})=\frac{c_{i-1}}{d_{i}}
.
\end{align}
The procedure is visualized in Figure~\ref{fig:esempioSemplice}. We consider the basic domain 
\[
Q_{i-1}=[0,a_{i-1}]\times[0,c_{i-1}]
\qquad i\in\N
\] 
and we partition it at the $i$-th step into $d_{i}$ vertical sub-strips
\[
S_{i} = [0,a_{i-1}]\times\left[0, c_{i}\left(1+e_{i}\right)\right]
\] 
of size $a_{i-1}\times c_{i}(1+e_{i})$: this is possible by the relation~\eqref{E:recreldec}.
We obtain the flowing partition:
\[
Q_{i-1}=[0,a_{i-1}]\times [0,c_{i-1}]=\bigcup_{j=0}^{ d_{i}-1} \left(j\rv_{i} +S_{i}\right)
\qquad
\rv_{i}=\left (0,\frac{c_{i-1}}{d_{i}}\right)=\left (0,c_{i}\left(1+e_{i}\right)\right) .
\]
The vertical strip $S_{i}$ is then partitioned into three horizontal strips as follows:
\begin{itemize}
\item The intermediate strip is the translation of a rectangle
\[
R_{i} =[0,b_{i}]\times \left[ 0,c_{i}\left(1+e_{i}\right)\right] 
.
\]
This strip provides space for smooth junctions among vertical curves.
\item The two extremal horizontal strips have size $a_{i} \times c_{i}(1+e_{i})$. Each extremal horizontal strip is the union of two vertical rectangles which are translations of 
\[
L_{i}  = [0,a_{i}]\times \left[0, c_{i}e_{i}\right]
\qquad
Q_{i} = [0,a_{i}]\times\left[ 0 , c_{i}\right]
\]
\end{itemize}
This is possible if $2{a_{i}}+b_{i}=a_{i-1}$.
One can define vectors for the translation and one can write
\begin{gather*}
\qr_{i}=\left(a_{i},0\right)
\qquad
\rr_{i}=\left(0,c_{i}e_{i}\right)
\qquad
\ur_{i}=(a_{i}+b_{i},0)
\qquad
\sr_{i}=\left(a_{i-1}, c_{i}(1+e_{i})\right)
\\
S_{i}=L_{i}\bigcup\left[\rr_{i}+Q_{i}\right] \bigcup [\qr_{i}+ R_{i} ]\bigcup \left[\ur_{i}+Q_{i} \right] \bigcup  \left[ \sr_{i}- L_{i} \right] 
.
\end{gather*}
Let us term $Q_{j_{1}\dots j_{i}}^{h_{1}\dots h_{i}}$ the rectangles generated at the $i$-th step, for $ h_{\ell}\in\{0,1\}$, $j_{\ell}\in\{0,\dots,d_{\ell}-1\}$:
\[
Q_{j_{1}\dots j_{i}}^{h_{1}\dots h_{i}}= \prom_{j_{1}\dots j_{i}}^{h_{1}\dots h_{i}} +Q_{i},
\quad 
\prom_{j_{1}\dots j_{i}}^{h_{1}\dots h_{i}} = \sum_{\ell=1}^{i}j_{\ell}\rv_{\ell} +\sum_{\ell=1}^{i} h_{j_{\ell}} \ur_{\ell} +  \sum_{\ell=1}^{i} (1-h_{j_{\ell}})\rr_{\ell}.
\]
Since at each step such rectangles are nested into previous ones, the translation vector $\prom_{j_{1}\dots j_{i}}^{h_{1}\dots h_{i}}$ takes into account which is the list of nesting rectangles: $j_{\ell}$ tells us in which strip of the $\ell$-th step we are it, $h_{j_{\ell}}$ tells us if at the $\ell$-th step we are in the low right triangle of the corresponding strip---in case $h_{j_{\ell}}=1$---or if we are instead in the upper left one---in case $h_{j_{\ell}}=0$.

The remaining regions of shape $L_{i}$ and $R_{i}$ are not partitioned anymore at future steps.

\step{Setting up the parameters}
Set for $i\in\N\cup\{0\}$
\[
a_{i}\doteq 2^{-i-1}(1+2^{-i}).
\]
We impose that for $i\in\N$ the interval of length $a_{i-1}$ is divided into three subintervals, two of which of length $a_{i}$ and one of length $b_{i}$:
\[
b_{i}\doteq a_{i-1}-2 a_{i}=2^{-i}(1+2^{-i+1}) -2\cdot2^{-i-1}(1+2^{-i}) = 2^{-i}[\cancel{1}+2^{-i+1}\cancel{-1}-2^{-i}]=2^{-2i}.
\]

We determine now values of $c_{i},d_{i},e_{i}$ satisfying~\eqref{E:recreldec} plus an additional requirement that we need later.
Example~\ref{Ex:raccordo} constructs a curve $x=\gamma(t)\in C^{2}([0, b_{i}])$ which satisfies  
\begin{subequations}
\label{E:caratteristicaesempio}
\begin{align}
&0\leq\ddt u\circ {(t,\gamma(t))} \leq 1 \qquad  u\circ {(t,\gamma(t))}\doteq \ddt\gamma(t) \qquad \forall t\in[a,b_{i}]
\\
&\ddt u(t,\gamma(t))\Big|_{t=0}=u(0,\gamma(0))=0=u(b_{i},\gamma(b_{i}))=\ddt u(t,\gamma(t))\Big|_{t=b_{i}} \ .
\end{align}
\end{subequations}
Easy computations, reported for completeness in Example~\ref{Ex:raccordo}, show the estimate
\begin{align*}
\max_{s,t\in[0,b_{i}]} |\gamma(t)-\gamma(s)| < 2f(b_{i}/2)=2^{-4i-1}.
\end{align*}
For every $0\leq c<2f(b_{i}/2) $, Example~\ref{Ex:raccordo} also constructs a $C^{2}$ curve satisfying~\eqref{E:caratteristicaesempio} and
\[
\gamma\left(b_{i}\right)-\gamma(0)=c. 
\]
Define therefore the following sequence of positive numbers, lower than $2f(b_{i}/2)$:
\begin{align*}
c_{i}\doteq \frac{2f(b_{i}/2)}{\Pi_{j=1}^{i}(1+e_{j})}=\frac{2^{-4i-1}}{\Pi_{j=1}^{i}(1+2^{-j})},
\quad c_{0}=1/2,\qquad
e_{i}\doteq 2^{-i},
\qquad i\in\N.
\end{align*}
Define finally the integer ratio
\begin{align*}
d_{i}=\frac{c_{i-1}}{c_{i}(1+e_{i})} =\frac{16\cdot \cancel{2^{-4i-1}}}{\cancel{\Pi_{j=1}^{i-1}(1+2^{-j})}}\frac{\cancel{(1+2^{-i})}\cancel{\Pi_{j=1}^{i-1}(1+2^{-j})}} { \cancel{2^{-4i-1}}}\frac{1}{\cancel{1+2^{-i}}}=16
.
\end{align*}

A table of the first values is the following
\[
\begin{array}[h]{r|c|c|c|c|c}
&a_{i} &{b_{i}}&c_{i} &d_{i}&e_{i}
\\
\hline
0 & 1& -&1/2& -&-\\
1 &{3}/{8}&{1}/{8}&1/48& 16&1/2\\
2  &{5}/{32}&{1}/{32}&1/960&16&1/4\\
\vdots&\vdots&\vdots&\vdots&\vdots&\vdots
\end{array}
\]

\step{Measure of the Cantor set}
We compute the measure of the set 
\[
K=\bigcap_{i\in\N}\bigcup_{j_{1}=0}^{d_{1}-1}\dots \bigcup_{j_{i}=0}^{d_{i}-1}\bigcup_{h_{1},\dots , h_{i}=0}^{1}Q_{j_{1}\dots j_{i}}^{h_{1}\dots h_{i}}
\]
Let us describe the above intersection step by step.
At the first step there are $d_{1}$ stripes translations of $S_{1}$, each of which generates two rectangles which are translation of $Q_{1}=[0,a_{1}]\times [0,c_{1}]$. Therefore
\[
\Ll^{2}\left(\bigcup_{j=0}^{d_{1}-1}\bigcup_{h=0}^{1} Q^{h}_{j} \right)= 2d_{1}c_{1}a_{1}\stackrel{\eqref{E:recreldec}}{=} 
   2 \frac{ c_{0}}{1+e_{1}}a_{1}.
\]
At the second step, each rectangle $j_{\ell}\rr_{1}+h_{\ell}\rr_{1} +Q_{0}$, for $h_{\ell}\in\{0,1\}$ and $j_{\ell}\in\{0,\dots, d_{1}-1\}$, produces $2d_{2}$ smaller rectangles of size $a_{2}\times c_{2}$: there are thus $2d_{1}\cdot 2d_{2}$ rectangles of size $a_{2}\times c_{2}$.
More generally, each rectangle $Q^{h_{1}\dots h_{i-1}}_{j_{1}\dots j_{i-1}}$ generates $2d_{i}$ rectangles each of size $c_{i}a_{i}$, and there are $2d_{1}\cdot\dots\cdot 2 d_{i-1}$ such rectangles. We can hence conclude that at the $i$-th step
\begin{align*}
\Ll^{2}\left(\bigcup_{j_{1}=0}^{d_{1}-1}\dots \bigcup_{j_{i}=0}^{d_{i}-1}\bigcup_{h_{1},\dots , h_{i}=0}^{1}Q_{j_{1}\dots j_{i}}^{h_{1}\dots h_{i}}\right)
&= 
2^{i} d_{1}\cdots d_{i-1}d_{i}c_{i} a_{i} 
\\
&\stackrel{\eqref{E:recreldec}}{=} 
2^{i}d_{1}\dots d_{i-1}\frac{c_{i-1}}{1+e_{i}}a_{i}
\\
&\stackrel{\eqref{E:recreldec}}{=} 
2^{i}a_{i} \frac{c_{0}}{\prod_{j=1}^{i}(1+e_{j})}
\end{align*}
As the series $\sum_{j=0}^{\infty} e_{j}$ converges, by the elementary estimate $\prod_{j=1}^{i}(1+e_{j})\leq \exp(\sum_{j=1}^{i} e_{j})$ the infinite product converges and in the limit we get
\begin{align}
\Ll^{2}(K)=\lim _{i\to\infty} \cancel{2^{i}}\cdot 2^{\cancel{-i}-1}(1+2^{-i}) \cdot \frac{2^{-1}}{\prod_{j=1}^{i}(1+2^{-j})} &=\frac{  2^{-3}}{\prod_{j=1}^{\infty}(1+2^{-j})} 
\notag
 \\
\label{E:misuraKprimo}
&>\frac{  2^{-3}}{\prod_{j=1}^{\infty}(1+2^{-j})^{2^j}}=\frac{1 }{8e} \ .
\end{align}
In particular $K$ is non-negligible.
We also observe that it is compact, since the $i$-th element of the intersection is the union of finitely many closed rectangles contained in $Q_{0}$.

\step{Assigning $u$ and characteristic curves}
We divided $Q_{0}$ into different regions in order to facilitate the definition of the characteristic curves. 
Set:
\begin{itemize}
\item $u\equiv 0$ in each region which is created at the $i$-th step as a translation of $L_{i}$, $i\in\N$.
\item define in $R_{i}$ characteristic curves providing smooth junctions, as in Example~\ref{Ex:raccordo}, from
\begin{align*}
\begin{array}{l}
u=0 \text{ on } \{0\}\times [0,c_{i}e_{i}]\\
u=0 \text{ on } \{0\}\times[c_{i}e_{i}, c_{i}(1+e_{i})]
\end{array} 
\text{ to } 
\begin{array}{l}
u=0 \text{ on } \{b_{i}\}\times[0,c_{i}]\\
u=0 \text{ on } \{b_{i}\}\times[c_{i},c_{i}(1+e_{i})].
\end{array} 
\end{align*}
We associated in this way characteristic curves, and therefore $u$, to each fundamental domain $R_{i}$.
Characteristic curves are defined in the region $\prom_{j_{1}\dots j_{i}}^{h_{1}\dots h_{i}}+\qr_{i}+R_{i}$, translation of $R_{i}$, as the above characteristic curves translated by the same vector $\prom_{j_{1}\dots j_{i}}^{h_{1}\dots h_{i}}+\qr_{i}$.
\end{itemize}
The dashed lines in the RHS of Figure~\ref{fig:esempioSemplice} give an idea of the qualitative behaviour.
We have
\[
u\in C^{1}(Q_{0}\setminus K) \cap C(Q_{0}).
\]
The unique continuous extension of $u$ to $Q_{0}$ vanishes on $K$.

\step{Conclusion} By~\eqref{E:misuraKprimo} the set $K$ has positive measure.
We now notice that every characteristic curve intersects $K$ in a single point, and countably many of them in two points.
Indeed, fix any $i\in\N$. The iterative construction is made in such a way that each characteristic curve intersecting a region $Q_{j_{1}\dots j_{i}}^{h_{1}\dots h_{i}}$ is uniquely defined out of it. In particular, if \emph{any} characteristic curve of the continuous function $u$ intersects a rectangle $Q_{j_{1}\dots j_{i}}^{h_{1}\dots h_{i}}$ it does not intersect in the complement of $Q_{j_{1}\dots j_{i}}^{h_{1}\dots h_{i}}$ other regions constructed as translation of any $Q_{i}$---with the exception of the curves on the boundary of  $Q_{j_{1}\dots j_{i}}^{h_{1}\dots h_{i}}$, which run on the boundary of another equal rectangle. This implies that the counter-image of $K$ by \emph{any} Lagrangian parameterization must have null measure, as each vertical section of this counterimage is made by the single point which is the intersection of $K$ with the relative characteristic composing the parameterization---or by two such points, for countably many curves.

We show in Figure~\ref{fig:liste} at a better scale the iterative horizontal subdivision of the height $a_{0}$ first in two extremal horizontal strips of height $a_{1}$ (blue ones) and a central strip of height $b_{1}$ (central one), then the subdivision of each horizontal strip of height $a_{1}$ into two horizontal strips of height $a_{2}$ (blue ones) and a central strip of height $b_{2}$, and so on al later iterations. The compact $K$ lies within blue regions.
\begin{figure}
\includegraphics[width=.4\linewidth]{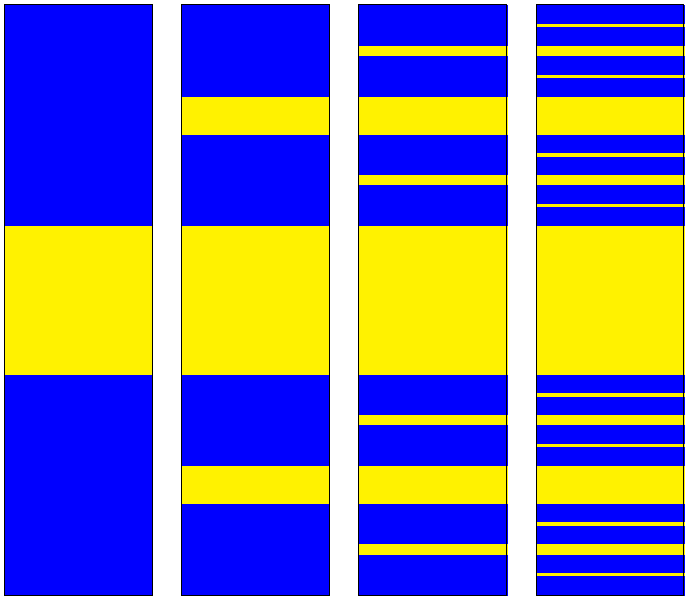}
\includegraphics[width=.4\linewidth]{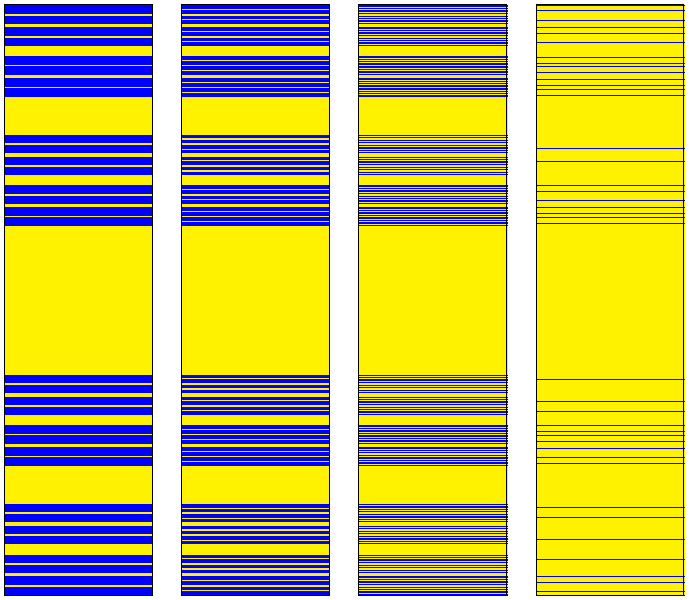}
\caption{From left to right, figures illustrate the iterative horizontal subdivision of the height $a_{0}$---left figure---first in two extremal horizontal strips of height $a_{1}$ (blue ones) and a central strip of height $b_{1}$ (central one), then---second figure---the subdivision of each horizontal strip of height $a_{1}$ into two horizontal strips of height $a_{2}$ (blue ones) and a central strip of height $b_{2}$, and so on al later iterations. $K$ lies within blue regions. The regions $L_{i}$ are so thin, even after two iterations, that they are not visible in such a picture.}
\label{fig:liste}
\end{figure}

\section{Non-negligible points of non-differentiability along characteristics}
\label{Ss:nondifferentiabilityset}

The example in this section shows the following: even when $u$ is Lipschitz continuous along characteristics and the flux $f$ is convex, there could be a compact, $\Ll^{2}$-positive measure set $K\subset\R^{2}$ of points where $u$ fails to be differentiable along characteristics, whichever characteristic curves one chooses through the point.
One can also have $u\in C^{\infty}(\R^{2}\setminus K)$, but clearly it will be just continuous on the whole region.
This provides as well a second example of non-absolute continuity of Lagrangian parametrizations, indeed this does not contradict the Lipschitz continuity of $u$ along any characteristic curves \cite[Theorem~30]{file1ABC}.
Such continuous solution $u$ is not H\"older continuous, for any exponent

The behavior in this section is prevented by the $\alpha$-convexity of the flux~\cite[Theorem~1.2]{estratto}: we give an example where the convex flux function vanishes at $0$ together with all its derivatives, while it is uniformly convex out of the origin.

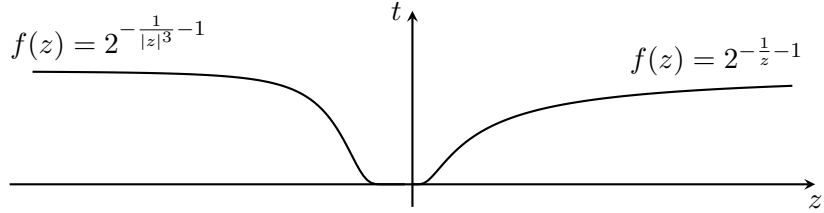
\begin{figure}[ht]
\begin{tikzpicture}[thick]\def\assexLung{5}
\def\asseyLung{2}
\draw[-stealth](-0.3-\assexLung,0)--(\assexLung+0.3,0) node[below]{$z$};
\draw[-stealth](0,-0.3)--(0,\asseyLung+0.3) node[left]{$t$};
\draw[samples=100, domain=0.0001:5,yscale=3]    plot (\x,{2^(-1/\x-1)}); \draw (4,1.3) node[above] {$f(z)=2^{-\frac{1}{z}-1}$};
\draw[samples=100, domain=-0.1:-5,yscale=3] plot (\x,{2^(-1/(-\x*\x*\x)-1)}); 
\draw (-4,1.5) node[above] {$f(z)=2^{-\frac{1}{|z|^3}-1}$};
\end{tikzpicture}
\caption{Flux function $f$ considered in \S~\ref{Ss:nondifferentiabilityset}. Close to the origin, $f$ is strictly convex, but not uniformly convex. This flux function is $C^{\infty}(\R)$, with all derivatives vanishing at the origin, but of course it is not analytic}
\label{F:convexfluxnotunif}
\end{figure}

Define the flux function, for $|z|\leq 1$, given by (Figure~\ref{F:convexfluxnotunif})
\[
f(z)=\begin{cases}2^{-\frac{1}{z}-1} &z\geq 0, \\ 2^{-\frac{1}{|z|^{3}}-1} &z<0. \end{cases}
\]
We mimic the construction of \S~\ref{S:Cantoparam}, modifying the regions $L_{i}$ and the parameters.

\firststep
\step{Setting up the parameters}
Similarly to \S~\ref{S:Cantoparam} we set up parameters $a_{i},b_{i},c_{i}$ satisfying
\begin{align*}
&a_{i-1}=2 a_{i}+b_{i}, &&
c_{i-1}=c_{i}(1+d_{i}), &&
c_{i}=2f(b_{i}/2) &&
i\in\N ,
\end{align*}
so that the following properties for the recursive construction are satisfied:
\begin{itemize}
\item An interval of length $a_{i-1}$ is the disjoint union of two intervals of length $a_{i}$ plus an interval of length $b_{i}$.
\item There exists a $C^{2}$ curve  $x=\gamma(t)$, for $t\in[0, b_{i}]$ which satisfies relations
\begin{subequations}
\label{E:caratteristicaesempio2}
\begin{align}
&0\leq\ddt u\circ {(t,\gamma(t))} \leq 1 \qquad  f'(u)\circ {(t,\gamma(t))}\doteq \ddt\gamma(t) \qquad \forall t\in[a,b_{i}]
\\
&\ddt u(t,\gamma(t))\Big|_{t=0}=u(0,\gamma(0))
=0=u(b_{i},\gamma(b_{i}))=\ddt u(t,\gamma(t))\Big|_{t=b_{i}} \ .
\end{align}
\end{subequations}
and
\begin{align*}
&\gamma\left(b_{i}\right)-\gamma(0)=\frac{c_{i}}{\lambda_{i+1}}=c_{i+1}, 
&& \lambda_{i}\doteq  1+d_{i}. 
\end{align*}
\end{itemize}
The last point is given again by Example~\ref{Ex:raccordo}. In particular, one can fix
\begin{subequations}
\begin{align}
\label{E:ai}
&a_{i}=2^{-i-1}(1+2^{-i})
&&b_{i}=2^{-2i}
&&a_{0}=1
\\
\label{E:estimateonmovingcurve2}
&c_{i}=2f(b_{i}/2)=\cancel{2}\cdot 2^{-\frac{1}{b_{i}/2}\cancel{-1}}= 2^{-2^{2i+1}},
&& c_{0}=\frac{1}{4},
&& i\in\N.
\\
& \lambda_{i}=\frac{c_{i-1}}{c_{i}}=\frac{2^{-2^{2i-1}}}{2^{-2^{2i+1}}}=2^{2^{2i-1}(4-1)}=2^{3\cdot2^{2i-1}},
\\
\label{E:rybrfffr}
& d_{i}=\frac{c_{i-1}-c_{i}}{c_{i}}=\lambda_{i+1}-1=2^{3\cdot2^{2i-1}}-1 \ .
\end{align}
\end{subequations}
Notice finally that the difference $c_{i-1} - c_{i}$ is asyntotic to $c_{i-1}$:
\[
\frac{c_{i-1} - c_{i} }{c_{i-1}}=1-\frac{c_{i} }{c_{i-1}}=1-\frac{1}{1+d_{i}}=1-\lambda_{i}^{-1}
\]
and thus
\begin{equation}
\label{E:iteration}
d_{i}c_{i}= c_{i-1} - c_{i}= c_{i-1}(1-\lambda_{i}^{-1}) \ .
\end{equation}
A table of the first values is the following
\[
\begin{array}[h]{r|c|c|c|c}
&a_{i} &{b_{i}}&c_{i} &d_{i}
\\
\hline
0 & 1& -&{1}/{4}& -\\
1 &{3}/{8}&{1}/{4}&{1/256}& 63\\
2  &{5}/{32}&{1}/{16}&2^{-32}&2^{24}-1\\
\vdots&\vdots&\vdots&\vdots&\vdots
\end{array}
\]
\begin{figure}[htp!]
\input{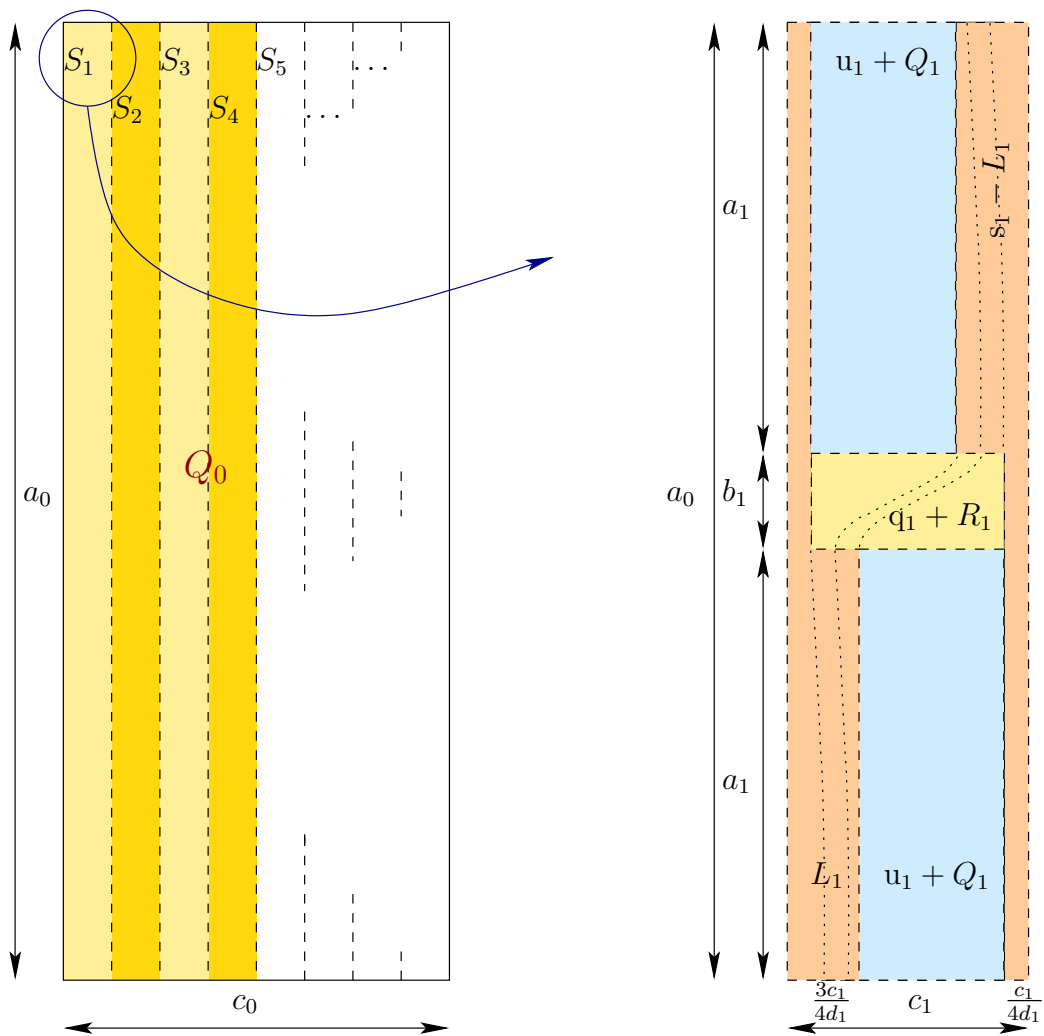}
\caption{The initial region $Q_{0}$ and one of its strips $S_{1}$. Proportions are distorted. {{}Dotted} lines suggest the qualitative behavior of characteristic curves. {}Dashed lines suggest the blocks $S_{i}$, $Q_{i}$, $L_{i}$, $R_{i}$ in the construction.}
\label{fig:esempioFig1}
\end{figure}

\step{Definition of the iterative regions (Figure~\ref{fig:esempioFig1})}
We consider the basic domain 
\[
Q_{i-1}=[0,a_{i-1}]\times[0,c_{i-1}]
\qquad i\in\N
\] 
and we partition it at the $i$-th step into $d_{i}$ vertical sub-strips which are translations of
\[
S_{i} = [0,a_{i-1}]\times\left[0, c_{i}\left(1+\frac{1}{d_{i}}\right)\right]
\]
each of size $a_{0}c_{1}(1+1/ d_{1})$: indeed when $i=1$ we have $a_{0}=1$ and
\[
d_{1}c_{1}(1+1/d_{1})=d_{1}c_{1}+c_{1}=c_{0}\ .
\]
We obtain the flowing partition of $Q_{0}$, that we write with a notation suitable for later iterations, when we will then have $i>1$:
\[
Q_{i-1}=[0,a_{i-1}]\times [0,c_{i-1}]=\bigcup_{j=0}^{ d_{i}-1} \left(j\rv_{i} +S_{i}\right)
\qquad
\rv_{i}=\left (0,\frac{c_{i-1}}{d_{i}}\right) .
\]
The vertical strip $S_{i}$ is then partitioned into three horizontal strips as follows, see Figure~\ref{fig:esempioFig1}: 
\begin{itemize}
\item The intermediate horizontal strip has size $b_{i}\times c_{i}(1+1/d_{i})$.
It is in turn made of three vertical sub-strips: 
\begin{itemize}
\item two extremal strips of size $b_{i}\times c_{i}/4d_{i}$, included in the regions $L_{i}$ of the next point, and
\item a central one strip which is the translation of a rectangle 
\[
R_{i} =[a_{i},a_{i}+b_{i}]\times \left[ 0,c_{i}\left(1+\frac{1}{2d_{i}}\right)\right] .
\]
This strip provides space for smooth junctions among vertical curves.
\end{itemize}
\item The two extremal strips have size $a_{i}\times c_{i}(1+1/d_{i})$. Each horizontal strip is in turn made of three vertical sub-strips: 
\begin{itemize}
\item a central iterative strip which is a translation of the rectangle
\[
Q_{i} = [0,a_{i}]\times\left[ 0 , c_{i}\right];
\]
\item two extremal strips which are translations and reflections of
\[
L_{i}  = [0,a_{i}]\times \left[0, \frac{3c_{i}}{4d_{i}}\right]\bigcup [0,a_{i-1}]\times\left[0, \frac{c_{i}}{4d_{i}}\right] .
\]
\end{itemize}
\end{itemize}
We can hence write the decomposition
\begin{gather*}
S_{i}=L_{i}\bigcup  \left[ \sr_{i}- L_{i} \right]\bigcup [\qr_{i}+ R_{i} ] \bigcup\left[\rr_{i}+Q_{i}\right]\bigcup \left[\ur_{i}+Q_{i} \right]
\\
\sr_{i}=\left(a_{0}, c_{1}+\frac{c_{1}}{d_{i}}\right)
\qquad
\qr_{i}=\left(0,\frac{c_{i}}{4d_{i}}\right)
\qquad
\rr_{i}=\left(0,\frac{3c_{i}}{4d_{i}}\right)
\qquad
\ur_{i}=\left(a_{i}+b_{i},\frac{c_{i}}{4d_{i}}\right)
.
\end{gather*}
We term $Q_{j_{1}\dots j_{i}}^{h_{1}\dots h_{i}}$, for  $h_{\ell}\in\{0,1\}$ and $j_{\ell}\in\{0,\dots,d_{\ell}-1\}$ when $\ell=1,\dots, i$, the rectangles that are generated at the $i$-th step: with the above notation one can write
\[
Q_{j_{1}\dots j_{i}}^{h_{1}\dots h_{i}}=\sum_{\ell=1}^{i}j_{\ell}\rv_{\ell} +\sum_{\ell=1}^{i} h_{j_{\ell}} \ur_{\ell} +  \sum_{\ell=1}^{i} (1-h_{j_{\ell}})\rr_{\ell} +Q_{i}.
\]
Each of this rectangles $Q_{j_{1}\dots j_{i}}^{h_{1}\dots h_{i}}$ will be further partitioned at the next step.
The remaining regions of shape $L_{i}, -L_{i}$ and $R_{i}$, suitably translated, are not partitioned any more.

\step{Measure of the Cantor set}
We compute the measure of the set 
\[
K=\bigcap_{i\in\N}\bigcup_{j_{1}=0}^{d_{1}-1}\dots \bigcup_{j_{i}=0}^{d_{i}-1}\bigcup_{h_{1},\dots , h_{i}=0}^{1}Q_{j_{1}\dots j_{i}}^{h_{1}\dots h_{i}}
\]
As each $Q^{h_{1}\dots h_{i-1}}_{j_{1}\dots j_{i-1}}$ generates $2d_{i}$ rectangles each of size $c_{i}a_{i}$, then at the $i$-th step
\begin{align*}
\Ll^{2}\left(\bigcup_{j_{1}=0}^{d_{1}-1}\dots \bigcup_{j_{i}=0}^{d_{i}-1}\bigcup_{h_{1},\dots , h_{i}=0}^{1}Q_{j_{1}\dots j_{i}}^{h_{1}\dots h_{i}}\right)
&= 
2^{i} d_{1}\cdots d_{i-1} d_{i}c_{i} a_{i} 
\\
&\stackrel{\eqref{E:iteration}}{=} 
2^{i}d_{1}\dots d_{i-1}[c_{i-1}(1-\lambda_{i}^{-1})]a_{i}
\\
&\stackrel{\eqref{E:iteration}}{=} 
2^{i}a_{i} c_{0}\prod_{j=0}^{i}(1-\lambda_{j}^{-1})
\end{align*}
As the series $\sum_{j=0}^{\infty} \lambda_{j}^{-1}$ converges, by the elementary estimate $\prod_{j=1}^{i}(1-\lambda_{j}^{-1})\leq \exp(-\sum_{j=1}^{i} \lambda_{j}^{-1})$ the infinite product also converges and in the limit we get
\begin{align}
\label{E:misuraK}
\Ll^{2}(K)&=\lim _{i\to\infty} \cancel{2^{i}}\cdot 2^{\cancel{-i}-1}(1+2^{-i}) \cdot \prod_{j=1}^{i}\left(1-2^{-3\cdot 2^{2i-1}}\right) \cdot 2^{-2} 
\\\notag
&=  \frac{1 }{8}\prod_{j=1}^{\infty}\left(1-2^{-3\cdot 2^{2i-1}}\right)  
\\\notag
&>  \frac{1 }{8}\prod_{j=1}^{\infty}\left(1-2^{-3\cdot 2^{2i-1}}\right)^{2^{3\cdot 2^{2i-1}}} =\frac{1 }{8e} \ .
\end{align}
In particular $K$ is non-negligible.

\step{Assigning $u$ and characteristic curves}
We subdivided $Q_{0}$ into different regions in order to facilitate the definition of the characteristic curves. 
We care now of defining
\[
u\in C^{1}(Q_{0}\setminus K) \cap C(Q_{0}).
\]
It will vanish on $K$ by continuity. We define simultaneously characteristic curves in $Q_{0}\setminus K$.
They will be defined separately in the different regions, and they will have smooth junctions.

Up to translations, focus on the fundamental regions $R_{i}$, $L_{i}$, $-L_{i}$.
We first specify the following common properties that characteristic curves should satisfy in a region $R\in\{R_{i}, L_{i}, -L_{i}\}$:
\begin{itemize}
\item Characteristic curves do not intersect.
\item Characteristic curves through points $(t,x)$ with $t\in\{0;a_{i};a_{i}+b_{i};a_{i-1}\}$ have there vertical tangent. 
This means that $u$ vanishes on those horizontal lines.
We impose moreover that at those points also the derivative of $u$ along the characteristic vanishes.
\item The image of the curves cover the whole region $R$, defining a $C^{1}(R)$ function
\[
u(t,x)\doteq \left(f'\right)^{-1}(\dot \gamma(t))
\quad\text{where $\gamma$ is the characteristic defined through $(t,x)\in R$}.
\]
There is no ambiguity in this structure due to the strict, even if not uniform, convexity of $f$.
\end{itemize}
We describe now the shape of the curves, depending on the region $R\in\{R_{i},L_{i},-L_{i}\}$.
The dashed lines in the RHS of Figure~\ref{fig:esempioFig1} give an idea of this qualitative behavior.
The precise expression of the characteristic curves that we describe can be computed by elementary auxiliary computations that we report for completeness in \S~\ref{Ss:auxcomputations}.

\firstsubstep
\substep{Region $R_{i}$ (Figure~\ref{fig:esempioFig1})}
The unique characteristic curve through a point in 
\begin{align*}
\begin{array}{l}
\{a_{i}\}\times [0,\frac{c_{i}}{4d_{i}}]\\
\{a_{i}\}\times[\frac{c_{i}}{4d_{i}},\frac{c_{i}}{2d_{i}}]\\
\{a_{i}\}\times[\frac{c_{i}}{2d_{i}},c_{i}+\frac{c_{i}}{2d_{i}}]
\end{array} 
\text{ reaches increasingly a point in } 
\begin{array}{l}
\{a_{i}+b_{i}\}\times[0,c_{i}]\\
\{a_{i}+b_{i}\}\times[c_{i},c_{i}+\frac{c_{i}}{4d_{i}}]\\
\{a_{i}+b_{i}\}\times[c_{i}+\frac{c_{i}}{4d_{i}},c_{i}+\frac{c_{i}}{2d_{i}}] .
\end{array} 
\end{align*}
This is compatible with the previous common requirements by the choice~\eqref{E:estimateonmovingcurve2}: an explicit construction of a curve joining $(a_{i}, \lambda \frac{c_{i}}{4d_{i}})$ and $(a_{i}+b_{i}, \lambda c_{i})$, for $\lambda\in[0,1]$, is provided by Example~\ref{Ex:raccordo}.

\substep{Region $L_{i}$ (Figure~\ref{fig:esempioFig1})}
The unique characteristic curve through a point in 
\begin{align*}
\begin{array}{lll}
\{0\}\times [0,e_{i}] &\text{ reaches decreasingly a point in }& \{a_{i}\}\times [0,\frac{c_{i}}{4d_{i}}]\\
\{0\}\times[e_{i},e_{i}+\frac{c_{i}}{4d_{i}}] &\text{ reaches decreasingly a point in }&\{a_{i}\}\times[\frac{c_{i}}{4d_{i}},\frac{c_{i}}{2d_{i}}]\\
\{0\}\times[e_{i}+\frac{c_{i}}{4d_{i}}, \frac{3c_{i}}{4d_{i}}] &\text{ reaches decreasingly a point in }&\{a_{i}\}\times[\frac{c_{i}}{2d_{i}},\frac{3c_{i}}{4d_{i}}]\\
\{a_{i}\}\times [0,\frac{c_{i}}{4d_{i}}] &\text{ remains constant up to }&\{a_{i-1}\}\times [0,\frac{c_{i}}{4d_{i}}]
\end{array} 
\end{align*}
for a value of $e_{i}$ that we specify now.
We are in a situation completely analogous to Example~\ref{Ex:raccordo}. 
We show it in the most interesting region, which is the second one, and we define together $e_{i}$.
Consider the characteristic $\gamma(t)=\gamma(t;a_{i}/8)$ given by~\eqref{E:curvagiunzione}, but substituting $f(-z)$ to $f(z)$: this corresponds to the fact that $u$ is decreasing along the curve instead of increasing, and thus $u$ is negative---which in turn corresponds to $\gamma$ decreasing.
Set then
\[
e_{i}=\frac{c_{i}}{4d_{i}}-[\gamma(a_{i}) -\gamma(0)]> \frac{c_{i}}{4d_{i}}.
\]
Notice that the intervals above are well defined, because $e_{i}<\frac{c_{i}}{2d_{i}}$:
\begin{align*}
|\gamma(a_{i})-\gamma(0)| &< f(-a_{i}/2) = 2^{-\frac{1}{(a_{i}/2)^{3}}-1}=2^{-\frac{2^{3i+6}}{(1+2^{-i})^{3}}-1} 
\\
&<2^{-2^{3i+3}}<2^{-2^{2i+3}}\\
&< \frac{ 2^{-2^{2i+1}}}{4(2^{3\cdot2^{2i-1}}-1)}=\frac{c_{i}}{4d_{i}}.
\end{align*}
We observe finally that $ u$ decreases along $\gamma(t)$ up to $t=a_{i}/2$ and its derivative is $-1$ for $t$ in $[a_{i}/8,3a_{i}/8]$: then the minimum value of $u$, reached at the centre of the interval, is
\begin{align}
\label{E:growthu}
u(\gamma(a_{i}/2))=-\frac{3a_{i}}{8}<- \frac{a_{i}}{4}.
\end{align}

\substep{Region $-L_{i}$ (Figure~\ref{fig:esempioFig1})}
The region $-L_{i}$ is entirely similar to the region $L_{i}$ already described, therefore we will be quick. We require that the unique characteristic curve through a point in 
\begin{align*}
\begin{array}{lll}
\{-a_{i}\}\times [-\frac{3c_{i}}{4d_{i}},-\frac{c_{i}}{2d_{i}}] &\text{ reaches decreasingly a point in }& \{0\}\times [-\frac{3c_{i}}{4d_{i}},-e_{i}-\frac{c_{i}}{4d_{i}}]\\
\{-a_{i}\}\times [-\frac{c_{i}}{2d_{i}},-\frac{c_{i}}{4d_{i}}] &\text{ reaches decreasingly a point in }&\{0\}\times[-e_{i}-\frac{c_{i}}{4d_{i}}, -e_{i}]\\
\{-a_{i}\}\times [-\frac{c_{i}}{4d_{i}},0]&\text{ reaches decreasingly a point in }&\{0\}\times[ -e_{i},0]\\
\{-a_{i-1}\}\times [-\frac{c_{i}}{4d_{i}},0] &\text{ remains constant up to }&\{-a_{i}\}\times [-\frac{c_{i}}{4d_{i}},0]
\end{array} 
\end{align*}
for the value of $e_{i}$ already defined.
Along characteristic curves passing through $\{-a_{i}\}\times [-\frac{c_{i}}{2d_{i}},-\frac{c_{i}}{4d_{i}}] $ the function $u$ reaches a minimum value which is less than $-a_{i}/4$.

\step{ $u$ on $K$ is not differentiable along characteristics} 
We remind that $K$ has positive measure by~\eqref{E:misuraK}.
We check now that $K$ is made of non-differentiability points of $u$ along characteristics.
Consider a point $(t,x)\in K$: it is the countable intersection of countable unions of rectangles, and the $2^i d_{1}\cdots d_{i}$ rectangles in the $i$-th union are translations of $Q_{i}$. With the notation above, one can write
\[
\forall i\in\N\ , \ \ell=1,\dots,d_{i-1}\quad \exists h_{\ell}^i\in\{0,1\},\ j_{\ell}^i\in\{0,\dots,d_{\ell}-1\} \quad:\quad (t,x)\in \bigcap_{i=1}^{\infty} Q_{j_{1}^i\dots j_{i}^i}^{h_{1}^i\dots h_{i}^i}.
\]
For simplicity of notation focus on $j_{\ell}^i=0$, $h_{\ell}^i=1$ for every $i,\ell\in\N$, $\ell\leq i$, which means
\[
(t,x)=\left(a_{0},\ \sum_{i\in\N}\frac{c_{i}}{4d_{i}}\right)\ .
\]
The general case is entirely analogous. Denote by $\gamma(t)$ a characteristic curve through $(t,x)$.
We show the non-differentiability by proving the following:
\begin{enumerate}
\item  There is a sequence of points $\{t^{0}_{i}\}_{i\in\N}$ converging to $t$ such that $u(t^{0}_{i},\gamma(t^{0}_{i}))=0$.
\item There is a sequence of points $\{t^{-}_{i}\}_{i\in\N}$ such that
\begin{equation}
\label{E:tempibrutti}
|t^{-}_{i}- t|=\frac{a_{i}}{2} \leq a_{i},\qquad u(t^{-}_{i},\gamma(t^{-}_{i}))=-\frac{3a_{i}}{8}< -\frac{a_{i}}{4}.
\end{equation}
This implies that $u\circ \gamma$ cannot have zero derivative at $t$: since $u(t,x)=0$ we get
\[
\liminf_{i\to\infty } \left|\frac{u(t,x)-u( t^{-}_{i}, \gamma(t^{-}_{i}))}{t-t^{-}_{i}}\right|
>
\liminf_{i\to\infty } \frac{a_{i}/4}{a_{i}}=\frac{1}{4}.
\]
\end{enumerate}
The two points together imply that $u\circ\gamma$ cannot be differentiable at $t$, because along the two different sequences $\{t^{-}_{i}\}_{i\in\N}$  and $\{t^{0}_{i}\}_{i\in\N}$  the different quotients have two different limits: respectively $0$ and something less than or equal to $-\frac{1}{4}$.
The two sequences are defined as follows. 
\begin{enumerate}
\item By construction (Figure~\ref{fig:esempioFig1}) $\gamma$ intersects each lower side of the rectangles $Q^{1}_{0}$, $Q^{11}_{00}$, $Q^{111}_{000}$, \dots at times $a_{0}-a_{1}$, $a_{0}-a_{2}$, \dots.
 On that side we set $u$ vanishing: then $u$ vanishes on the sequence of times $t^{0}_{i}=a_{0}-a_{i}$, $i\in\N$, which converges to $t=a_{0}$. 
\item Characteristic curves were conveyed to the lower side of $Q^{1\dots 1}_{0 \dots 0}$, which is say a translation of $Q_{i}$, from a specific part of the region
\[
\left(\sum_{j=1}^{i-1}(a_{j}+b_{j}), \ \sum_{j=1}^{i-1} \frac{c_{j}}{4d_{j}} \right)+ L_{i}.
\]
There are times inside this region translated of $L_{i}$ where our requirement is satisfied, like
\[
t^{-}_{i}\doteq \sum_{j=1}^{i-1}(a_{j}+b_{j})+\frac{a_{i}}{2}=a_{0}-b_{i}-\frac{3a_{i}}{2} 
\] 
works by~\eqref{E:growthu}.
Notice that $|t-t^{-}_{i}|=  b_{i}+3a_{i}/2<a_{i-1}\to 0$.
\end{enumerate}
This concludes proof that $u$ at any point of $K$ is not differentiable along characteristics, and hence this concludes the example.

\begin{remark}\label{R:noHolder}
We notice that the function $u$ constructed in the present section is not H\"older continuous: in the same setting where~\eqref{E:tempibrutti} was derived one has
\[
u(t_{i}^{-},\gamma(t_{i}^{-}))=-\frac{3a_{i}}{8}\stackrel{\eqref{E:ai} }{=}-\frac{3 }{8}\cdot 2^{-i-1}(1+2^{-i}) \qquad \sim \qquad-3\cdot 2^{-i-4} \ .
\]
Moreover, $u$ vanishes on the left side of each $L_{i}$ which contains part of $\gamma$, so that if we denote by $x^{*}_{i}$ the intersection of that left side with the fixed time $t=t_{i}^{-}$, namely $x^{*}=\sum_{j=1}^{i-1} \frac{c_{j}}{4d_{j}}$, then
\[
0<
\gamma(t_{i}^{-}) - x^{*}_{i}
<
\frac{3 }{4}\frac{c_{i}}{d_{i}}
\stackrel{\eqref{E:estimateonmovingcurve2}-\eqref{E:rybrfffr}}{=}
\frac{3 }{4}\frac{ 2^{-2^{2(i+1)}}}{2^{3\cdot2^{2i-1}}-1 }
 \qquad \sim \qquad
\frac{3 }{4}\cdot 2^{ - {11}\cdot2^{2i-1}} \ .
 \]
We thus conclude that for every constant $\alpha>0$
\[
\lim_{i}\frac{|u(t_{i}^{-},\gamma(t_{i}^{-}))-u(t_{i}^{-},x^{*})|}{|\gamma(t_{i}^{-}) - x^{*}_{i}|^{\alpha}}
=
\lim_{i}\frac{\cancel3\cdot 2^{-i-2}}{\cancel3\cdot 2^{ - {11\alpha}\cdot2^{2i-1} }}
=+\infty \ .
\]
\end{remark}

\section{Failure of Lipschitz continuity along characteristics}
\label{Ss:Notlipanlongchar}

A continuous distributional solution $u$ to
\begin{equation*}
\tag{\ref{EE}}
\pt u(t,x) + \px (f(u(t,x))) =  \gfr(t,x)
\qquad f\in C^{2}(\R),
\qquad 
|\gfr(t,x)|\leq G
\end{equation*}
is always Lipschitz continuous along characteristics if 
\begin{align}
\tag{\ref{ass:h}}
\Ll^{1}(\clos({\infl(f)}))=0
\end{align}
is satisfied \cite[Theorem~30]{file1ABC}. Our first aim is to show now that this assumption is needed: we provide a flux function $f$ with non-negligible inflection points and a continuous (Lagrangian and Eulerian) solution  $u(t,x)\equiv u(x)$ which is not Lipschitz continuous when restricted to \emph{some} characteristic curves. The notion of Broad solution does not make sense for such fluxes.
Moreover, changing the Lagrangian parameterization, the Lagrangian source $\gfr$ might change.

Example~\ref{Ex:Lagrnobroad} shows that for some Lagrangian parameterization it does not make sense defining a Lagrangian source, since Lipschitz continuity of $u$ on its characteristics might fail. It also proves that, whenever one chooses a `good' Lagrangian parameterization, the Lagrangian source should be fixed accordingly: no universal choice is possible even within admissible Lagrangian parameterizations. This is even more astonishing considering that the Eulerian source term is constantly $1$ and $\pt u=0$.

\begin{remark}
\label{Rem:1}
We first remind that there might be continuous solutions to~\eqref{EE}, even in the autonomous case $u(t,x)\equiv u(x)$, which are Cantor-like functions.
Consider a flux function $f\in C^{2}(\R)$ which is strictly increasing and which satisfies
\[
S\doteq \{z \ : \ f'(z)=0\}\subset[0,1]\  ,
\qquad \Ll^{1}(S)>0\ ,
\quad S\text{ does not contain intervals}
.
\]
Notice that $f''(z)=0$ at all points $z\in S$ and all points of $S$ are inflection points of $f$, so that condition~\eqref{ass:h} is violated. Consider the function
\[
w(z)=z-\Ll^{1}(\{q<z\ : \ f'(q)=0\}) \ .
\]
First observe that $w$ is strictly increasing and $1$-Lipschitz continuous:
\[
z_{1}< z_{2}\qquad\Rightarrow \qquad 0<w(z_{2})-w(z_{1})=z_{2}-z_{1}-\Ll^{1}([z_{1},z_{2}]\cap S)\leq z_{2}-z_{1} 
\]
 since $(z_{1},z_{2})\setminus S $ is a non-empty and open subset of $[z_{1},z_{2}]$. Being $ w'(z)=0$ and $f'(z)=0$ when $z$ is a Lebesgue point of $S$, then by the area formula
\begin{equation}
\label{E:grgrgrgwg}
\Ll^{1}(w(S))=0 \ ,
\qquad
\Ll^{1}(f(S))=0 \ .
\end{equation}
This vanishing condition implies that the (continuous, strictly increasing) inverse $w^{-1}$ of $w$ is a Cantor-Vitali like function, since $w^{-1}$ maps the $\Ll^{1}$-negligible set $w(S)$ to the non $\Ll^{1}$-negligible set $S$. Set
\[
u(t,x)=w^{-1}(x) \ .
\]
Notice that $u$ defines a continuous function which is constant in $t$ and strictly increasing in $x$. The composition $f\circ u(t,x)=\int_0^x(f'(w^{-1}))$ is Lipschitz continuous and $u$ is a distributional solution of
\[
\pt u(t,x) + \px (f(u(t,x))) = \gfr(t,x)\qquad\text{with $\gfr(t,x)= f(w^{-1})'(x)$} .
\]
\end{remark}
\nomenclature{$\uno_S$}{If $S$ is a set, $\uno_S$ is the function valued $1$ on $S$ and $0$ elsewhere}

\begin{example}[A distributional solution is not necessarily broad]
\label{Ex:Lagrnobroad}
Consider the same flux as in Remark~\ref{Rem:1}, where $f$ is the inverse of a Cantor-Vitali-like function.

Define the continuous function
\begin{equation}
\label{E:fnkurgbgrrg}
u(t,x) = f^{-1}(x),
\end{equation}
which is a distributional solution of the equation with continuous Eulerian source
\begin{equation}\label{E:PDEbanale}
\pt u +\px f(u)\equiv \px x=1.
\end{equation}
We now show that $u$ is not a broad solution, because it is not Lipschitz continuous on every characteristic curve.
We then compute that it is indeed a Lagrangian solution, as it must be~\cite[Corollary~46]{file1ABC}.

\firststep

\step{$u$ is not a broad solution} Consider the increasing and $1$-Lipschitz continuous function
\begin{align*}
w&&:&&z && \mapsto && z-\Ll^{1}(S\cap [0,z]).&&
\end{align*}
The derivative of $w$ at the density points of $S$ is $0$, while it is $1$ at the Lebesgue points of the complement.
The set $S$ is mapped into a $\Ll^{1}$-negligible set on which the singular part of $\partial_{z}w$ is concentrated: in particular, $w^{-1}$ is not absolutely continuous.
The following curve is well defined, because $w$ is a bijection:
\begin{equation}
\label{E:buttare}
\gamma(t)=f(w^{-1}(t)).
\end{equation}
Form the definition of $S$, $f'$ vanishes on it. As a consequence of~\eqref{E:grgrgrgwg} the characteristic $\gamma$ is Lipschitz continuous because the composition of $f$ and $w^{-1}$ is absolutely continuous, and
\[
\dot\gamma(t)=f'(w^{-1}(t))\stackrel{}{{=}}f'({{}f^{-1}}({{}f(w^{-1}(t)))})\stackrel{\eqref{E:fnkurgbgrrg}-\eqref{E:buttare}}{=}f'({{}u}(t,{{}\gamma(t)})).
\]
Nevertheless, $u$ is not absolutely continuous on this characteristic $\gamma$:
\[
 {{}u}(t,{{}\gamma(t)}) \stackrel{\eqref{E:fnkurgbgrrg}}{{=}}{{}f^{-1}}({{}\gamma(t)}) \stackrel{\eqref{E:buttare}}{{=}}{{} f^{-1}}({{}f(w^{-1}(t))}) = w^{-1}(t).
\]

\step{$u$ is a Lagrangian solution} For each $\tau\in \R$ the function
\begin{subequations}
\label{E:lspaorejff}
\begin{equation}
\chi(t,\tau)= f(t+\tau)
\end{equation}
is a characteristic:  by~\eqref{E:fnkurgbgrrg} and by definition of the Lagrangian parameterization
\[
{{}u}(t,{{}\chi(t,\tau)})={{}f^{-1}}({{}f(t+\tau)})=t+\tau
\]
and thus
\[
\pt \chi(t,\tau) = f'(t+\tau)=f'(u(\chi(t,\tau))) \ .
\]
Since $\chi$ is smooth and also monotone in $\tau$, $\chi$ is a Lagrangian parameterization associated with $u$.
Notice that the Lagrangian source coincides with the natural representative of the distributional one:
\begin{equation}
\ddt u(\chi(t,\tau))=\ddt(t+\tau)=1
\qquad
\forall t>0,\tau\in\R.
\end{equation}
\end{subequations}
This yields to the natural definition $\gfr(t,x)=1$ for $(t,x)\in\R^{+}\times\R$.

\step{$u$ admits incompatible Lagrangian sources}
On the compact, $\Ll^{1}$-negligible set of $x$ defined by $f(S)$ one has also the vertical characteristics.
One can clearly define a Lagrangian parameterization $\widetilde\chi$ which includes such characteristics. Being $u$ identically zero on these, then it would be necessary to define pointwise the source term
\[
\widetilde\gfr(t,x)=0
\qquad \forall (t,x)\in \R\times f(S).
\]
This Borel regular function is different from the source of the Lagrangian parameterization~\eqref{E:lspaorejff} and different from the continuous representative of the Eulerian source~\eqref{E:PDEbanale}.
However, being $\Ll^{1}(f(S))=0$, the functions $\gfr$ and $\widetilde\gfr$ differ on an $\Ll^{2}$-negligible set. The Lagrangian sources of the two Lagrangian parameterizations identify the same distributional source:
\[
 \widetilde\gfr = \gfr
 \qquad \Ll^2\text{-a.e.}.
\]
Attaining different values on the vertical characteristics $\{(t,\tau)\}$ for $\tau\in f(S)$, $\gfr$ fails to be a source associated to the Lagrangian parameterization $\widetilde\chi$ as well as $\widetilde \gfr$ fails to be a source associated to the Lagrangian parameterization $\chi$.
Indeed, there are positive measure times $t\in ( S-\tau)$ where the sources differ on the characteristics~\eqref{E:lspaorejff}:
$1=\gfr(t,\chi(t,\tau))\neq \gfr(t,\chi(t,\tau))=0$.

We conclude observing that Lagrangian sources relative to $\chi$ and to $\widetilde \chi$ are not compatible: there exists no Borel, bounded function $\gfr$ such that
\begin{itemize}
\item $\gfr\circ i_{\chi}=1$ $\Ll^2$-a.e. and 
\item which satisfies $\gfr(t,x)=0$ for $\Ll^{1}$-a.e.~$t$ if $x\in f(S)$.
\end{itemize}

\end{example}

\section{Auxiliary computations}
\label{Ss:auxcomputations}
We collect here elementary computations which exhibit junctions, with a generic flux $f$, among characteristics defined in separate regions.
Each characteristic $\gamma$ must satisfy \[\dot \gamma(t)=f'(u({(t,\gamma(t))}))\] by definition.
In particular, if one prescribes $u({(t,\gamma(t))})$ smooth enough then $\gamma$ is determined, up to translations, simply by integrating in time the prescriber function $f'(u({(t,\gamma(t))}))$.
We see below some examples: $u$ growing linearly, quadratically, and a combination of the two.
\begin{example}
\label{Ex:uffa}
Let us start with a trivial example.
Consider a characteristic $\gamma$ such that 
\[
\ddt u(t,\gamma(t))=1 \text{ for }t\in(a,b). 
\]
In this case, the characteristic $\gamma$ can be easily computed. Indeed, one has
\[
u(t,\gamma(t))= u_{a}+t-a,
\text{ where $u_{a}\doteq u(a,\gamma(a))$,}
\]
and therefore for $t\in[a,b]$
\begin{align*}
\gamma(t)-\gamma(a)&=\int_{a}^{t}\dot\gamma(s)ds=\int_{a}^{t}f'(u(s,\gamma(s)))ds
=\int_{a}^{t}f'(u_{a}+s-a)ds
\\
&=f(u_{a}+t-a)- f(u_{a}).
\end{align*}
If more generally 
\[
\ddt u(t,\gamma(t))= v(t)>0 \text{ for }t\in(a,b), 
\]
then \[t\mapsto u(t,\gamma(t))=u_{a}+\int_{a}^{t}v(s)ds, \qquad u_{a}=u(a,\gamma(a))\] is invertible with inverse that we denote $U^{-1}$.
Even if the expression of the characteristic $\gamma$ is not as explicit as before, one can compute the variation of $\gamma(t)$ from time $a$ to time $t$ by \begin{align}
\label{E:formulastupida}
\gamma(t)-\gamma(a)&=\int_{a}^{t}\dot\gamma(s)ds=\int_{a}^{t}f'(u(s,\gamma(s)))ds = \int_{a}^{t}f'\left(u_{a}+\int_{a}^{s}v(r)dr\right)ds
\\
&
\stackrel{z=u(s,\gamma(s))}{=}\int_{u_{a}}^{u_{t}}\frac{f'(z)}{ v(U^{-1}(z))}dz \notag
,
\end{align}
where we termed \[u_{a}\doteq u(a,\gamma(a)), \qquad u_{t}\doteq u(t,\gamma(t))=u_{a}+\int_{a}^{t}v(s)ds.\]
\end{example}

\begin{example}
\label{Ex:riuffa}
Let us consider another easy example where the trace of $u$ on a characteristic $\gamma$ determines the characteristic $\gamma$ itself.
If 
\[
u({(t,\gamma(t))})=\frac{t^{2}}{2\tau}\text{ for } t\in(0,\tau)
\quad \Rightarrow\quad
v(t)\doteq  \ddt u({(t,\gamma(t))}=\frac{t}{\tau},\ v(U^{-1}(z))=\sqrt{2z/\tau}
\]
and equation~\eqref{E:formulastupida} gives us
\begin{align*}
\gamma(\tau)-\gamma(0)&=\int_{0}^{\tau/2}\frac{f'(z)}{ \sqrt{2z/\tau}}dz,
\qquad
\ddt u(0,{\gamma}(0))=0,\quad \ddt u(\tau, {\gamma}(\tau))=1
.
\end{align*}
Similarly, if one fixes
\[
u({(t,\gamma(t))})=C-\frac{(t-b)^{2}}{2\tau}
\qquad\text{for $t\in(b-\tau,b)$},
\]
then
\[
v(t)=(b-t)/\tau
\qquad\text{and}\qquad
v(U^{-1}(z))=-\sqrt{2(C-z)/\tau} \ .
\]
\end{example}

\begin{example}
\label{Ex:raccordo}
Suppose one requires that $0\leq\ddt u({(t,\gamma(t))})\leq 1$ for $t\in[0,b]$ and
\begin{align}
\label{E:curvamodello}
\gamma(b)-\gamma(0) = c,
\qquad
\ddt u(t, {\gamma}(t))\Big|_{t=0}=u(0,{\gamma}(0)) =0=u(b,{\gamma}(b))=\ddt u(t,{\gamma}(t))\Big|_{t=b}.
\end{align}
We describe the characteristic in the interval $[0,b/{2}]$, then we take the symmetric
\begin{subequations}
\label{E:curvagiunzione}
\begin{gather}
\gamma(t) = \gamma(b/2)+\int_{b-t}^{b/2}\dot\gamma(s)ds \qquad t\in (b/2,b].
\end{gather}
The (positive) values of $c$ one can hope to achieve are less than $2f(b/2)$ since
\begin{align*}
\gamma(b/2)-\gamma(0)&=\int_{0}^{b/2}\dot\gamma(s)ds=\int_{0}^{b/2}f'(u(s,\gamma(s)))ds
\stackrel{0\leq u(s,\gamma(s)) <s}{<}\int_{0}^{b/2}f'(s)ds  = f(b/2)
.
\end{align*}
We give below $C^{2}$-curves achieving each precise value $0<c<2f(b/2)$, distinguishing $c$ small or big.

\emph{Case $c\sim 2f(b/2)$.}
Combine Examples~\ref{Ex:uffa} and~\ref{Ex:riuffa}: the continuous function
\begin{align}
\gamma(t;\tau)=
\begin{cases}
\gamma(0)+\int_{0}^{t^{2}/(2\tau)}\frac{f'(z)}{ \sqrt{2z/\tau}}dz
& 0\leq t\leq \tau
\\
\gamma(\tau)+f(\tau/2+t-\tau)- f(\tau/2)
& \tau< t\leq b/2-\tau
\\
\gamma(b/2-\tau)+ \int_{0}^{b/2 -\tau-(t-b/2)^{2}/(2\tau)}f'(z)\sqrt{\frac{\tau}{b-2\tau -2z}}dz
& b/2-\tau< t\leq b/2
\end{cases}
\end{align}
\end{subequations}
will satisfy the requirements in~\eqref{E:curvamodello} for one fixed $\tau\in (0,b/4]$, provided that
\[
f\left(\frac{b}{2}\right)>\frac{c}{2} \geq \gamma\left(\frac{b}{2};\frac{b}{4}\right)-\gamma(0)=2\int_{0}^{b/2}\frac{f'(z)}{ 2\sqrt{2z/b}}dz .
\]
This choice is equivalent to assigning the $C^{1}([0,b/2])$ function
\[
u(i_{\gamma_{\tau}}(t))
=
\begin{cases}
\frac{t^{2}}{2\tau}
& 0\leq t\leq \tau \ ,
\\
\frac{\tau}{2 }+t-\tau\equiv t-\frac{\tau}{2}
& \tau< t\leq \frac{b}{2}-\tau \ ,
\\
\frac{b}{2}-\tau-\frac{(t-\frac{b}{2})^{2}}{2\tau}
& \frac{b}{2}-\tau< t\leq \frac{b}{2} \ ,
\end{cases}
\]\[
\ddt u(i_{\gamma_{\tau}}(t))
=
\begin{cases}
\frac{t }{ \tau}
& 0\leq t\leq \tau \ ,
\\
1
& \tau< t\leq \frac{b}{2}-\tau \ ,
\\
 \frac{b}{2\tau}- \frac{  t }{\tau }
& \frac{b}{2}-\tau< t\leq \frac{b}{2} \ .
\end{cases}
\]

\emph{Case $c$ small.} If instead $c$ is small one may have
\[
f\left(\frac{b}{2}\right) > 2\int_{0}^{b/2}\frac{f'(z)}{ 2\sqrt{2z/b}}dz>  \frac{c}{2} \ .
\]
In this case one can just consider, for the suitable $\tau\in [0,2/b]$, the choice of the $C^{1}([0,b/2])$ function
\[
u(i_{\gamma_{\tau}}(t))
=
\begin{cases}
\tau t^{2} 
&  0\leq t\leq \frac{b}{4} \ ,
\\
\frac{\tau b^{2}}{8}-\tau \left(t-\frac{b}{2}\right)^{2} 
&  \frac{b}{4} <t\leq \frac{b}{2} \ ,
\end{cases}
\]\[
\ddt u(i_{\gamma_{\tau}}(t))
=
\begin{cases}
2\tau t  
&  0\leq t\leq \frac{b}{4} \ ,
\\
(b-2t)    \tau 
&  \frac{b}{4} <t\leq \frac{b}{2} \ .
\end{cases}\]
This defines the characteristic, for example in the first half interval $[0,b/4]$,
\[
\gamma(t;\tau)
=
\gamma(0)+\int_{0}^{\tau t^{2}}\frac{f'(z)}{ 2\sqrt{\tau z}}dz
\qquad
 0\leq t\leq b/4
.\]
Since $\gamma(b/2;\tau)\downarrow \gamma(0)$ as $\tau\downarrow 0$ and $\gamma(b/2;1/b)>c/2$, by the continuity there is indeed a suitable $\tau$.
\end{example}

\nomenclature{$\Haus^{1}$, $\Haus^{2}$}{{}1- or 2-dimensional Hausdorff measure}
\nomenclature{$\Ll^{1}$, $\Ll^{2}$}{1- or 2-dimensional Lebesgue measure}
\nomenclature{$u$, $f$}{$u$ is a fixed continuous solution for the balance law~\eqref{EE} and $f$ is the $C^2$-flux}
\nomenclature{$\lambda(t,x)$}{The composite function $f'(u(t,x))$}
\nomenclature{$\pt,\px$}{Partial derivatives in the sense of distributions}
\nomenclature{$\ddt$}{Classical derivative in the real variable $t$}
\nomenclature{$X$}{Subset either of $\R^{2}$ or of $\R$, usually Borel.}
\nomenclature{$\Omega$}{Open subset of either $\R^2$ or $\R$, if needed connected}

\section{A remark on continuous sources}
\label{S:cotinuoussources}

We emphasize the nontrivial fact that for continuous sources the continuous representative is both a \emph{particular} Lagrangian source and a good Eulerian source, if $f$ has negligible inflection points.
While for $\alpha$-convex fluxes of~\cite{estratto} a continuous Eulerian source is also a Broad source, this fails when $f(u)=u^{3}$, see Remark~\ref{Rem:cubico}.
\begin{theorem}
\label{T:continousSourceGen}
If $u,\gfr\in C(\R^2)$ and $f\in C^2(\R)$, consider the following conditions:
\begin{enumerate} 
\item \label{item:primo} The distribution $\pt u(t,x) + \px (f(u(t,x))) $ identifies the continuous Borel function $\gfr$.
\item \label{item:secondo} There exists a family of characteristic curves dense in $\R^2$ along which $u$ is Lipschitz continuous and the classical derivative of $u$ identifies the continuous Borel function $\gfr$.
\end{enumerate}
Then~\eqref{item:secondo}$\Rightarrow$\eqref{item:primo} always. If inflection points of $f$ are {{}isolated}, then~\eqref{item:primo}$\Rightarrow$\eqref{item:secondo}.
\end{theorem}

\begin{remark}\label{Rem:cubico}
We recall that~\eqref{item:secondo} is always~\cite[\S~A.1]{file1ABC} equivalent to
\begin{enumerate}\setcounter{enumi}{2}
\item \label{item:duebis} There exists a Lagrangian parameterization with the continuous Borel function $\gfr$ as associated Lagrangian source.
\end{enumerate} 
One could desire the stronger condition
\begin{enumerate}\setcounter{enumi}{3}
\item \label{item:terzo}  the classical derivative of $u$ along every characteristic curve along which $u$ is Lipschitz continuous identifies the continuous function $\gfr$.
\end{enumerate} 
{{}In Example~\ref{Ex:cubico}} considering $f(u)=u^3$ and $u(t,x)=\sqrt[3]x$ we realize that there is no hope: the continuous representative of the source term, if any, is not necessarily the Broad source, neither the Lagrangian one for all parameterizations.

{}In Theorem~\ref{T:continousSourceGen} assuming that inflection points of $f$ are isolated is not sharp, nevertheless it covers this meaningful and surprising case, still having a simple proof.
\end{remark}

\begin{example}\label{Ex:cubico}
Consider $f(u)=u^3$, $f'(u)=3u^{2}$ and $u(t,x)=\sqrt[3]x$ then
\[
f(u(t,x))=x
\qquad\text{so that}\qquad \pt u(t,x) + \px (f(u))= \px x\equiv 1
\]
 but $\gamma(t)\equiv 0$ is a characteristic, and $\ddt u(t,\gamma(t))\equiv 0\neq1$: the Broad source is the Borel function $\gfr(t,x)\doteq \uno_{\{x\neq0\}}$ and not the continuous function identically $1$.
 
 In particular, the continuous function identically $1$ is not the Lagrangian source for the Lagrangian parameterization $\chi:\R\times(-2,2)\to \R$ given by
 \[
 \chi(t,y)=
 \begin{cases}
 \left( t-\setttgh (y-1) \right)^{{}3} & \text{if }t\geq \setttgh (y-1)\,,\ y\in (0,2)\,,
 \\
 0& \text{if }t< \setttgh (y-1)\,,\ y\in (0,2)\,,
 \end{cases}
 \]
 and $\chi(t,y)=-\chi(-t,-y)$ if $-2<y<0$, while $\chi(t,0)=0$ for all $t$.
In the half-plane $x>0$, setting $y= 1+\tanh t_{0}$, $s=t-t_{0}$, characteristic curves are
\[
s\mapsto \left(t_{0}+ s,  s^{{}3}\right)
\qquad t_{0}\in\R\,, \ s>0\,,
\quad 
{}u(t,\gamma_{t_{0}}(t))=\sqrt[3]{(t-t_{0})^{3}}=t-t_{0}\,.
\]
\end{example}

\begin{proof}
\eqref{item:secondo}$\Rightarrow$\eqref{item:primo}. 
Consider the construction in the proof of \cite[Lemma 27]{file1ABC}.
The Lagrangian source terms $\gx_k$ of the $\BV$ approximations $u_k$ there defined are not continuous, but where they do not vanish they satisfy $|\gx_k-\gx|\leq \omega(\delta_n)$, being $\omega$ the modulus of uniform continuity of the Lagrangian source $\gx$ and $\delta_k=\max_{t\in[0,1]}\max_{j=1,\dots,k}\{|\bar\gamma_j(t)-\bar\gamma_{j-1}(t)|\}$. 
In particular, since the measure of the region where $\gx_k$ vanishes `because of cuts' converges to $0$ as $k\uparrow\infty$, we have that $\gx_k$ converges to the Lagrangian source $\gx$ of $u$ in $L^1$ when $\gx$ is continuous.
Since by~\cite[Lemma 22]{file1ABC} $\gx_k$ satisfies also $\pt u_k+\px f(u_k)=\gx_k$, and $u_k$ converges uniformly to $u$, then of course the countinuous function $\gx$ is an Eulerian source term for $u$.

\eqref{item:primo}$\Rightarrow$\eqref{item:secondo}. {{}Case 1}. The case when the flux is convex follows from passing to the limit in the first inequality of~\cite[(3.1b)]{file1ABC}, and in the analogous opposite one, which is immediate when the Eulerian source $ \gfr$ is continuous{{}: $\gfr$ is both broad and Eulerian}.

{{}Case 2: Single inflection point}.
By the previous case, one can identify source terms in the complement of $u^{-1}(\clos(\infl(f)))$, which is an open subset of $\R^2$. 
{{}Just to simplify notations, we assume that $0$ is the only inflection point of $f$---just consider $\widetilde f(u)=f(u+\overline u)$---with $f$ convex on $\{\overline u>0\}$ and concave on $\{u<0\}$; one can reduce to this situation considering $\widehat u(t,x)=u(t,-x)$ and $\widehat f(u)=-f(u)$.

Notice that $u$ is nondecreasing along characteristics lying in $\{ \gfr >0\}$, strictly increasing in $\{ \gfr >0\}\setminus u^{-1}(0)$.
It has the opposite monotonicity in $\{ \gfr <0\}$. 
In particular, if $\Haus^1\left(i_{\gamma}(\R)\cap u^{-1}(0)\cap\{ \gfr \neq 0\}\right) >0$ then the image of the characteristic $\gamma$ contains in $\{\gfr>0\}$ a single nontrivial segment with slope $f'(0)$; such segments can be parametrized by $r_{q}(t)=f'(0)t+q$ with $q\in Q$, $t\in(a_{q},b_{q})$: considering~\eqref{E:ecco}, by Fubini-Tonelli theorem necessarily $\Ll^{1}(Q)=0$ so that $Q$ has empty interior.

Consider a time $\overline t\in[a_{q},b_{q}]$, with $q\in Q$.
We claim that $u$ must change sign in any neighborhood of $P=(\overline t,r_{q}(\overline t))$. Suppose, for example $u$ is nonnegative in a neighborhood of $P$: then, in such neighborhood, $u$ is a solution of the same balance law with the convex flux $\widetilde f(u)=f(|u|)$.
Nevertheless, in Case 1 for a convex flux we proved that the continuous Eulerian source $\gfr\neq0$ must be the derivative of $u$ along all characteristics: this contradicts $u$ vanishing on the segment parametrized by $r_{q}$.

When $\gfr(P)\neq 0$, one can assume that $\gfr$ does not vanish in a neighborhood of $P$, for example fix $\gfr\geq\ell>0$ in a square $S$ centered in $P$ with side $4\rho>0$.

Since $u$ must change sign in any neighborhood of $P$, consider sequences $(t_{k}^{+}, x_{k}^{+})$, $(t_{k}^{-}, x_{k}^{-})$ converging to $(t,r_{q}(t))$ with $u(t_{k}^{+}, x_{k}^{+})>0$ and $u(t_{k}^{-}, x_{k}^{-})<0$.
Since $\ddt u\circ i_{\gamma}=\gfr\geq\ell>0$ for all characteristics $\gamma$ lying in $S\setminus u^{-1}(0)$, then for every characteristic $\gamma_{k}^{+}$ through $(t_{k}^{+}, x_{k}^{+})$ and $\gamma_{k}^{-}$ through $(t_{k}^{-}, x_{k}^{-})$ we have
\begin{equation}\label{E:eccoUffo}
u\circ i_{\gamma_{k}^{+}}(t_{k}^{+}+\rho)\geq u\circ i_{\gamma_{k}^{+}}(t_{k}^{+})+\ell\rho\,,\qquad
 u\circ i_{\gamma_{k}^{-}}(t_{k}^{-})\geq u\circ i_{\gamma_{k}^{-}}(t_{k}^{-}-\rho)+\ell\rho
\end{equation}
provided that $(t_{k}^{+}, x_{k}^{+})$, $(t_{k}^{-}, x_{k}^{-})$ are sufficiently close to $P$.

By Ascoli-Arzlea,  and by continuity of $u$ and $f'(u)$, the sequences $\{\gamma_{k}^{+}\}_{k\in\N}$ and $\{\gamma_{k}^{-}\}_{k\in\N}$ admit subsequences converging to limit characteristics $\gamma^{+}$, $\gamma^{-}$ through $P$ locally uniformly. We can exhibit the following characteristic through $P$:
\[
\overline \gamma(t)=\begin{cases}\gamma^{+}(t)&t\geq \overline t\,,\\ \gamma^{-}(t)&t\leq \overline t\,,\end{cases}
\qquad\text{so that } i_{\overline\gamma}(I)\cap u^{-1}(0)\cap\{ \gfr >0\}=\{P\}
\]
where $I$ is the connected component of $i_{\overline\gamma}^{-1}(\{ \gfr \neq 0\})$ which contains $\overline t$.
Since we are considering $f$ convex in $\{u>0\}$ and concave in $\{u<0\}$ then $\overline \gamma$ lies on the right of the line parametrized by $r_{q}$ for $t>\overline t$, and on the left for $t<\overline t$.
Because of this analysis, define more in general through every point $P$ the characteristic
\begin{equation}\label{E:goodCharact}
\gamma(t)
=
\begin{cases}
\max\{\gamma(t)\ :\ \gamma\text{ characteristic with }\gamma(\overline t)=\overline x  \} &\text{if }t\geq \overline t
\\
\min\{\gamma(t)\ :\ \gamma\text{ characteristic with }\gamma(\overline t)=\overline x  \} & \text{if }t\leq \overline t
\end{cases}
\end{equation}
for $t\in\R$. It is as required in item~\eqref{item:secondo}, as $i_{\overline\gamma}(\R)\cap u^{-1}(0)\cap \{ \gfr \neq 0\}$ is discrete.

When $f$ is convex in $\{u<0\}$ and concave in $\{u>0\}$ the expression of `good' characteristics is similar but maximum and minimum in~\eqref{E:goodCharact} are exchanged.
}

Case 3. Suppose {{} $ \infl(f)$ is discrete}. 
Lemma~\ref{L:setNull} proves that the Eulerian source $\gfr$ vanishes at $\Ll^2$-Lebesgue points of $u^{-1}(\clos(\infl(f)))$, so that
\begin{equation}\label{E:ecco}
\Ll^{2}\left(u^{-1}(\clos(\infl(f)))\cap\{ \gfr \neq0\}\right) =0\,.
\end{equation}
Lemma~\ref{L:setNull} also proves that the Lagrangian source term vanishes at $\Haus^1$-Lebesgue points of $(u\circ i_{\gamma })^{-1}(\clos(\infl(f)))$, for every characteristic curve $i_{\gamma}$. In particular, denoting by $\gfr$ the Eulerian source, also by Case 1, the Broad and the Eulerian source terms are the same $\Haus^1$-a.e.~on characteristic curves $i_{\gamma}:I\to\R^{2}$ such that
\begin{equation}\label{E:soifowgfehgrtherge}
\Haus^1\left(i_{\gamma}(I)\cap u^{-1}(\clos(\infl(f)))\cap\{ \gfr \neq0\}\right) =0\,.
\end{equation}
We now prove that such characteristic curves are dense in the plane, thus~\eqref{item:secondo} holds.

{}
Since we are assuming that $ \infl(f)$ is discrete, up to a rescaling fix e.~g.~that 
\begin{itemize}
\item $f$ is concave in $(2i- \sfrac{1}{2}, 2i+\sfrac{1}{2})$, $m_{i}=i+ \sfrac{1}{2}$ are inflection points, and 
\item $f$ is convex in $(2i+ \sfrac{1}{2}, 2i+\sfrac{3}{2})$, for $i\in\Z$.
\end{itemize}
We construct the desired characteristic through any point $(\bar t,\bar x)$ by approximation.

For $h,n\in\N$ consider times $t_{n,h}^{+}=\overline t+(h-1)/n$ and $t_{n,h}^{-}=\overline t-(h-1)/n$.
Define $\gamma_{n}(\overline t)=\gamma_{n}(t_{n,0}^{+})=\gamma_{n}(t_{n,0}^{-})=\overline x$.
Once defined $x_{n}^{+}\doteq\gamma_{n}(t_{n,h}^{+})$ and $x_{n}^{-}\doteq\gamma_{n}(t_{n,h}^{-})$, set
\begin{itemize}
\item in $(t_{n,h}^{+},t_{n,h+1}^{+}]$, for $h\in\N$: denoting $u_{n,h}^{+}\doteq u(t_{n,h}^{+},x_{n,h}^{+})$
\[
\gamma_{n}(t)
=
\begin{cases}
\max\{\gamma(t)\ :\ \gamma\text{ characteristic with }\gamma_{n}(t_{n,h}^{+})= x_{n,h}^{+} \} & \text{if }[\![u_{n,h}^{+}]\!] \text{ even,}
\\
\min\{\gamma(t)\ :\ \gamma\text{ characteristic with }\gamma_{n}(t_{n,h}^{+})= x_{n,h}^{+} \} & \text{if }[\![u_{n,h}^{+}]\!] \text{ odd}.
\end{cases}
\]
We denoted by $[\![\cdot]\!]$ the integer part of a number.
\item in $[t_{n,h+1}^{-},t_{n,h}^{-})$, for $h\in\N$: denoting $u_{n,h}^{-}\doteq u(t_{n,h}^{-},x_{n,h}^{-})$
\[
\gamma_{n}(t)
=
\begin{cases}
\min\{\gamma(t)\ :\ \gamma\text{ characteristic with }\gamma_{n}(t_{n,h}^{-})= x_{n,h}^{-} \} & \text{if }[\![u_{n,h}^{-}]\!] \text{ even,}
\\
\max\{\gamma(t)\ :\ \gamma\text{ characteristic with }\gamma_{n}(t_{n,h}^{-})= x_{n,h}^{-} \} & \text{if }[\![u_{n,h}^{-}]\!] \text{ odd}.
\end{cases}
\]
\end{itemize}
By compactness, one can extract a limit characteristic curve $\overline\gamma$ satisfying $\overline\gamma(\overline t)=\overline x$.

Recall that connected components of $\R\setminus\N$ contain a single inflection point of $f$.
When $i_{\gamma_{n}}([t_{n,h},t_{n,h+\ell}])$ lies within some $u^{-1}(\R\setminus\N)\cap\{\gfr\neq0\}$ then $i_{\overline\gamma}([t_{n,h},t_{n,h+\ell}])\cap u^{-1}(\inf(f))\cap\{ \gfr \neq0\}$ is at most a point by the analysis of Case~2.
Let $R>0$.
When $\sfrac2n\cdot\lVert f'(u)\rVert_{L^\infty([-R,R]^{2})}$ is less than the distance in the square $[-R,R]^{2}$ among the closed sets $u^{-1}(\N)$, $u^{-1}(\inf(f))$, then $\gamma_{n}$ has discrete intersection with $u^{-1}(\inf(f))\cap\{ \gfr \neq0\}$: for $n$ large enough, depending on $R>0$, then
\[
u(t,\gamma_{n}(t))-u(s,\gamma_{n}(s))=\int_{s}^{t}\gfr(\tau,\gamma_{n}(\tau)\,d\tau
\qquad \forall s,t\in[-R,R].
\]
By the local uniform convergence of $\gamma_{n_{j}}$, we conclude that $\overline\gamma$ is as wanted in~\eqref{item:secondo}.

\end{proof}

We stress that Counterexample~\ref{Ex:Lagrnobroad} shows that if the flux has inflection points of positive measure the correspondence from Eulerian to Lagrangian sources, even when both well defined, is not perfect regardless of the continuity of the source: differently from the case of negligible inflection points, there is no Borel function which works for two given different Lagrangian parameterizations, even when we are considering the case of continuous Eulerian source terms.
It is not clear if the weaker conditions~\eqref{item:secondo}-\eqref{item:duebis} still hold even for fluxes with non-negligible inflection points, when the Eulerian source is continuous.

\vskip\baselineskip
We also collect here properties of the solution, depending on the assumptions:
\\[1.2\baselineskip]
\begin{tabular}{p{4.2cm}|c|c|c}
&$\ell$-convexity & Negligible inflections & General case\\
\hline
absolutely continuous La\-gran\-gian pa\-ra\-me\-terization & \ding{55} (\S~\ref{S:Cantoparam}) & \ding{55} & \ding{55}\\
\hline
$u$ H\"older continuous & \ding{51} (\cite[Th.~1.2]{estratto})&\ding{55} (\S~\ref{Ss:nondifferentiabilityset})& \ding{55} \\
\hline
$u$ $\Ll^{2}$-a.e.~differentiable along characteristic curves &\ding{51}(\cite[Th.~1.2]{estratto}) &\ding{55} (\S~\ref{Ss:nondifferentiabilityset})& \ding{55}\\
\hline
$u$ Lipschitz continuous along characteristic curves & \ding{51}& \ding{51} \cite[Th.~30]{file1ABC}& \ding{55} (\S~\ref{Ss:Notlipanlongchar})\\
\hline
entropy equality  & \ding{51} & \ding{51} & \ding{51}(\cite[Lemma~42]{file1ABC})\\
\hline
compatibility of sources  & \ding{51} \ding{51} (\cite[Th.3.1]{estratto}) & \ding{51} (\S~\ref{S:compatibilitysources})& 
\end{tabular}
\vskip\baselineskip

The picture of $\ell$-convex fluxes is similar to the one of $\ell$-nonlinear fluxes \cite{CMP}, where one assumes that the flux $f$ has at each point a non-vanishing derivative of order between $2$ and $\ell$.
This is the case of analytic fluxes.

\printnomenclature


	%
	%

\section*{Acknowledgements}
	%
	%
The authors are partially supported by the Gruppo Nazionale per l’Analisi Matematica, la Probabilit\`a e le loro Applicazioni (GNAMPA) of the Istituto Nazionale di Alta Matematica (INdAM), and by the PRIN
2020 ``Nonlinear evolution PDEs, fluid dynamics and transport equations: theoretical foundations and application''.
L.~C.~is partially supported by the PRIN PNRR P2022XJ9SX of the European Union -- Next Generation EU.

	%
	%
	%
	%
\bibliographystyle{plain}

	%
	%
	%
	%
\vskip .5 cm

{\parindent = 0 pt\footnotesize
G.A.
\\
Dipartimento di Matematica, 
Universit\`a di Pisa,
largo Pontecorvo~5, 
56127 Pisa, 
Italy 
\\
e-mail: \texttt{giovanni.alberti@unipi.it}

\bigskip
S.B.
\\
SISSA, 
via Bonomea 265, 
34136 Trieste, 
Italy 
\\
e-mail: \texttt{stefano.bianchini@sissa.it}

\bigskip
L.C.
\\
Dipartimento di Matematica, 
Universit\`a di Padova,
via Trieste~63, 
35121 Padova,
Italy 
\\
e-mail: \texttt{laura.caravenna@unipd.it}

}

\end{document}